\documentclass[a4paper,centertags,oneside,12pt]{amsart}
\usepackage{mathrsfs}
\usepackage{appendix}
\usepackage{amssymb}
\usepackage{fancyhdr}
\usepackage{charter}
\usepackage{typearea}
\usepackage{pdfsync}
\usepackage{amsmath,amstext,amsthm,amscd}
\usepackage{dsfont}
\usepackage{mathrsfs}
\usepackage[a4paper,top=2cm,bottom=2cm,left=2cm,right=2cm]{geometry}

\usepackage{enumitem}
\setlist[1]{itemsep=5pt}
\newcommand{\comment}[1]{}

\makeatletter
      \def\@setcopyright{}
      \def\serieslogo@{}
      \makeatother

\newcommand{\Complex}{\mathbb C}

\newcommand{\ddbar}{\overline\partial}

\newcommand{\norm}[1]{\left\Vert#1\right\Vert}
\newcommand{\abs}[1]{\left\vert#1\right\vert}

\newcommand{\set}[1]{\left\{#1\right\}}
\newcommand{\To}{\rightarrow}

\newtheorem{theorem}{Theorem}[section]
\newtheorem{lemma}[theorem]{Lemma}
\newtheorem{corollary}[theorem]{Corollary}
\newtheorem{proposition}[theorem]{Proposition}
\newtheorem{question}[theorem]{Question}
\newtheorem{definition}[theorem]{Definition}

\newtheorem{remark}[theorem]{Remark}
\numberwithin{equation}{section}
\begin{document}
\title[]{Morse inequalities for Fourier components of Kohn-Rossi cohomology of
CR manifolds with $S^1$-action}
\author[]{Chin-Yu Hsiao}
\address{Institute of Mathematics, Academia Sinica, 6F, Astronomy-Mathematics Building,
No.1, Sec.4, Roosevelt Road, Taipei 10617, Taiwan}
\thanks{The first-named author was partially supported by Taiwan Ministry of Science of Technology project
103-2115-M-001-001 and the Golden-Jade fellowship of Kenda Foundation. }
\email{chsiao@math.sinica.edu.tw or chinyu.hsiao@gmail.com}
\author[]{Xiaoshan Li}
\address{School of Mathematics
and Statistics, Wuhan University, Hubei 430072, China \& Institute of Mathematics, Academia Sinica, 6F, Astronomy-Mathematics Building,
No.1, Sec.4, Roosevelt Road, Taipei 10617, Taiwan}
\thanks{The second-named author was  supported by Central university research Fund 2042015kf0049, Postdoctoral Science Foundation of China 2015M570660 and NSFC No. 11501422}
\email{xiaoshanli@whu.edu.cn or xiaoshanli@math.sinica.edu.tw}

\setlength{\headheight}{14pt}
\pagestyle{fancy}
\lhead{\itshape{Chin-Yu Hsiao \& Xiaoshan Li}}
\rhead{\itshape{Morse inequalities on CR manifolds with $S^1$ action }}
\cfoot{\thepage}

\begin{abstract}
Let $X$ be a compact connected  CR manifold of dimension $2n-1, n\geq 2$ with a transversal CR $S^1$-action on $X$. We study the Fourier components of the Kohn-Rossi cohomology with respect to the $S^1$-action. By studying the Szeg\"o kernel of the Fourier components we establish the Morse inequalities on $X$. Using the Morse inequalities we have established on $X$ we prove that there are abundant CR functions on $X$ when $X$ is weakly pseudoconvex and strongly pseudoconvex at a point.
\end{abstract}

\maketitle \tableofcontents
\section{Introduction}
The problem of embedding CR manifolds is prominent in areas such as complex analysis, partial differential equations and differential geometry. Let $X$ be a compact CR manifold of dimension $2n-1$, $n\geq 2$. When $X$ is strongly pseudoconvex and dimension of $X$ is greater than or equal to five, a classical theorem of
L. Boutet de Monvel~\cite{BdM1:74b} asserts that $X$ can be globally CR embedded into $\Complex^N$, for some $N\in\mathbb N$.
For a compact strongly pseudoconvex CR  manifold of dimension greater than or equal to five, the dimension of the kernel of the tangential Cauchy-Riemmann operator $\ddbar_b$ is infinite and we can find many CR functions to embed $X$ into complex space.
The classical example of non-embeddable three dimensional strongly
pseudoconvex CR manifold appears implicitly in the non-fillable example of
pseudoconcave manifold by Grauert \cite{G94}, Andreotti-Siu \cite{AS70} and Rossi \cite{Ro65} and was explicited by Burns \cite{B79}.
In \cite{Le92} it is shown  that a compact strongly pseudoconvex three dimensional CR manifold which admits an inner $S^1$-action is the boundary of a compact strongly pseudoconvex surface. By Kohn's result \cite{K86}, this implies that it is embeddedable in $\mathbb C^N$ for some $N$ (see~\cite{HM14I} for another proof). Bland obtained in \cite{Bl91} that for  a CR manifold which admits a free transversal $S^1$-action will be embedded into complex space if the CR structure admits a normal form relative to this $S^1$-action which has no
negative Fourier coefficients. Epstein ~\cite[Theorem\,A16]{Ep92} proved that  a three dimensional compact strongly pseudoconvex CR manifold $X$ with a global free transversal CR $S^1$ action can be embedded into $\mathbb C^N$ by the positive Fourier components of CR functions. Since the action is globally free, Epstein considered the quotient of the CR manifold by the $S^1$-action. The action which is CR and transversal implies that the quotient $X/S^1$ is a compact Riemann surface with a positive holomorphic line bundle. Then  $X$ is CR isomorphism to the circle bundle with respect to the dual bundle of the positive line bundle. Using Kodaira's embedding theorem, Epstein got the embedding theorem of the CR manifold by the space of positive  Fourier components of CR functions.

Motivated by Epstein's work, we will consider a compact CR manifold $X$ of dimension ${\rm dim} X=2n-1, n\geq 2$ with a transversal CR $S^1$-action and study the  Fourier components of Kohn-Rossi cohomology of $\overline\partial_b$-complex on $X$. The transversal CR $S^1$-action need not to be globally free but locally free. We use $T$ to denote the global vector field induced by the $S^1$-action. For $m\in\mathbb Z$ and $m>0$, we use $H^0_{b, m}(X)=\{u\in C^\infty(X): \overline\partial_b u=0, T u=imu\}$ to denote the $m$-th positive Fourier component of CR functions (see \cite{Ep92}). The embeddability of $X$ by positive Fourier components of CR functions is related to the behavior of the $S^1$ action on $X$. For example, if one can find $f_1\in H^0_{b,m}(X),\ldots,f_{d_m}\in H^0_{b,m}(X)$ and $g_1\in H^0_{b,m_1}(X),\ldots,g_{h_{m_1}}\in H^0_{b,m_1}(X)$ such that the map
\[\Phi_{m, m_1}:x\in X\to (f_1(x),\ldots,f_{d_m}(x),g_1(x),\ldots,g_{h_{m_1}}(x))\in\Complex^{d_m+h_{m_1}}\]
is a CR embedding. The $S^1$-action on $X$ induces naturally a $S^1$-action on $\Phi_{m, m_1}(X)$ and this $S^1$-action on $\Phi_{m, m_1}(X)$ is simply given by the following:
\[e^{i\theta}\circ(z_1,\ldots,z_{d_m},z_{d_m+1},\ldots,z_{d_m+h_{m_1}})=(e^{im\theta}z_1,\ldots,e^{im\theta}z_{d_m},e^{im_1\theta}z_{d_m+1},\ldots,e^{im_1\theta}z_{d_m+h_{m_1}}).\]
Thus, if one can embed such CR manifold by positive Fourier components of CR functions, we can describe the $S^1$-action explicitly. To study the embedding problem of such CR manifold by positive Fourier components of CR functions, it is crucial to be able to know
\begin{question}\label{j4}
When ${\rm dim}H^0_{b, m}(X)\approx m^{n-1}$ for $m$ large?
\end{question}
Inspired by Demailly's holomorphic Morse inequalities on complex manifolds \cite{D85}, \cite{D91}, \cite{MM07} and the recent works of the first-named author and Marinescu in \cite{HM12}, Hsiao \cite{H15}, \cite{H14} and Hsiao-Li \cite{HL15} on the Morse inequalities and Grauert-Riemenschneider criterion on CR manifolds, we obtain the Morse inequalities for the Fourier components of Kohn-Rossi cohomology of $\overline\partial_b$-complex. See Theorem~\ref{a4} and Theorem~\ref{a2} for the main results.

By the Morse inequalities we have obtained, we will show that a compact weakly pseudoconvex CR manifolds which admit a transversal CR locally free $S^1$-action will have abundant CR functions if it has a point where the Levi-form is strongly pseudoconvex (see Theorem~\ref{a5} for the details). This gives an answer of Question \ref{j4}.

\subsection{Set up and terminology}

Let $(X, T^{1,0}X)$ be a compact connected CR manifold of dimension $2n-1, n\geq 2$, where $T^{1,0}X$ is the given CR structure on $X$. That is, $T^{1,0}X$ is a subbundle of the complexified tangent bundle $\mathbb{C}TX$ of rank $n-1$ , satisfying $T^{1,0}X\cap T^{0,1}X=\{0\}$, where $T^{0,1}X=\overline{T^{1,0}X}$, and $[\mathcal V,\mathcal V]\subset\mathcal V$, where $\mathcal V=C^\infty(X, T^{1,0}X)$.

We assume that $X$ admits a $S^1$-action: $S^1\times X\rightarrow X, (e^{i\theta}, x)\rightarrow e^{i\theta}\circ x$. Here, we use $e^{i\theta}~(0\leq\theta<2\pi)$ to denote the $S^1$-action. Set $X_{\rm reg}=\{x\in X: \forall e^{i\theta}\in S^1, ~\text{if}~e^{i\theta}\circ x=x,~ \text{then}~e^{i\theta}=\rm id\}$. We call $x\in X_{{\rm reg}}$ a regular point of the $S^1$-action and complements of $X_{\rm reg}$ exceptional points. For every $k\in\mathbb N$, put
\begin{equation}\label{e-gue150614}
X_k:=\set{x\in X: e^{i\theta}\circ x\neq x, \forall\theta\in(0,\frac{2\pi}{k}), e^{i\frac{2\pi}{k}}\circ x=x}.\end{equation}
Thus, $X_{\rm reg}=X_1$. In this paper, we always assume that $X_{\rm reg}\neq\emptyset.$ By the Orbit type stratification (see Theorem 1.30 in \cite{Me03}), there are only finite $X_k's$ denoted by $X_1, X_{k_1}, \cdots, X_{k_p}$ which are not empty subset of $X$ such that $X=X_1\cup X_{k_1}\cup\cdots\cup X_{k_p}.$

Let $T\in C^\infty(X, TX)$ be the global real vector field induced by the $S^1$-action given as follows
\begin{equation}
(T u)(x)=\frac{\partial}{\partial\theta}\left(u(e^{i\theta}\circ x)\right)\Big|_{\theta=0}, u\in C^\infty(X).
\end{equation}
\begin{definition}\label{f1}
We say that the $S^1$-action $e^{i\theta} ~(0\leq\theta<2\pi$) is CR if
\begin{equation}\label{x1}
[T, C^\infty(X, T^{1,0}X)]\subset C^\infty(X, T^{1,0}X),
\end{equation}
where $[, ]$ is the Lie bracket between the smooth vector fields on $X$.
Furthermore, we say that the $S^1$-action is transversal if for each $x\in X$,
\begin{equation}\label{x2}
\mathbb CT(x)\oplus T_x^{1,0}(X)\oplus T_x^{0,1}X=\mathbb CT_xX.
\end{equation}
\end{definition}

\begin{remark}
The $S^1$-action on $X$ is said to be a locally free group action if $T(x)\neq0$ for every $x\in X$.
By (\ref{x2}), $T(x)$ will not vanish at any point $x\in X$, thus
the transversal CR $S^1$-action defined in Definition\ref{f1} is a locally free group action. For the knowledge of group action, we refer readers to \cite{Me03}, \cite{FR00}. The classical example of compact CR manifolds with transversal CR $S^1$-action is the circle bundle with respect to a Hermitian line bundle over a compact complex manifold. However, there are many examples of compact CR manifolds with transversal CR $S^1$-action which are not circle bundle.

For example, let $X=\{(z_1, z_2)\in\mathbb C^2: |z_1|^2+|z_1^2+z_2|^2+|z_2|^2=1\}$ which is a compact CR manifold with the following transversal CR $S^1$-action $$X\times S^1\rightarrow X, (z_1, z_2)\rightarrow(e^{i\theta}z_1, e^{2i\theta}z_2).$$ The $S^1$-action defined above is locally free and free on a dense, open, connected open subset $\{(z_1, z_2)\in X: z_1\neq 0\}$.
\end{remark}

In general, we have the following

\begin{lemma}\label{g2}
Let $X$ be a compact connected CR manifold with transversal CR locally free $S^1$-action.  Then $X_{\rm reg}$ is an open, dense subset of $X$. Moreover, the measure of $X\setminus X_{\rm reg}$ is zero.
\end{lemma}

The proof of Lemma~\ref{g2} is a direct corollary of Proposition 1.24 in  \cite{Me03} and similar results can be found in \cite{Du82}.  For the convenience of readers, we will prove Lemma~\ref{g2} in the appendix.

We assume throughout that $(X, T^{1,0}X)$ is a compact connected CR manifold with a transversal CR locally free $S^1$-action and we denote by $T$ the global vector field induced by the $S^1$-action. Let $\omega_0\in C^\infty(X,T^*X)$ be the global real $1$-form determined by $\langle\,\omega_0\,,\,u\,\rangle=0$, for every $u\in T^{1,0}X\oplus T^{0,1}X$ and $\langle\,\omega_0\,,\,T\,\rangle=-1$.

\begin{definition}\label{d-1.2}
For $x\in X$, the Levi-form $\mathcal L_x$ associated with the CR structure is the Hermitian quadratic form on $T_x^{1,0}X$ defined as follows. For any $U, V\in T_x^{1,0}X$, pick $\mathcal U, \mathcal V\in C^\infty(X, T^{1,0}X)$ such that $\mathcal U(x)=U, \mathcal V(x)=V$. Set
\begin{equation}
\mathcal L_x(U, \overline V)=\frac{1}{2i}\langle[\mathcal U, \overline{\mathcal V}](x), \omega_0(x)\rangle
\end{equation}
where $[,]$ denotes the Lie bracket between smooth vector fields. Note that $\mathcal L_x$ does not depend on the choice of $\mathcal U$ and $\mathcal V$.
\end{definition}
\begin{definition}
The CR structure on $X$ is called pseudoconvex at $x\in X$ if $\mathcal L_x$ is positive semidefinite. It is called strongly pseudoconvex at $x$ if $\mathcal L_x$ is positive definite. If the CR structure is (strongly) pseudoconvex at every point of $X$, then $X$ is called a (strongly) pseudoconvex CR manifold.
\end{definition}

Denote by $T^{\ast 1,0}X$ and $T^{\ast0,1}X$ the dual bundles of
$T^{1,0}X$ and $T^{0,1}X$, respectively. Define the vector bundle of $(0,q)$-forms by
$T^{\ast 0,q}X=\Lambda^qT^{\ast0,1}X$. Let $D\subset X$ be an open subset. Let $\Omega^{0,q}(D)$
denote the space of smooth sections of $T^{\ast0,q}X$ over $D$ and let $\Omega_0^{0,q}(D)$
be the subspace of $\Omega^{0,q}(D)$ whose elements have compact support in $D$.

Fix $\theta_0\in [0, 2\pi)$. Let
$$d e^{i\theta_0}: \mathbb CT_x X\rightarrow \mathbb CT_{e^{i\theta_0}x}X$$
denote the differential map of $e^{i\theta_0}: X\rightarrow X$. By the property of transversal CR $S^1$ action, we can check that
\begin{equation}\label{a}
\begin{split}
de^{i\theta_0}:T_x^{1,0}X\rightarrow T^{1,0}_{e^{i\theta_0}\circ x}X,\\
de^{i\theta_0}:T_x^{0,1}X\rightarrow T^{0,1}_{e^{i\theta_0}\circ x}X,\\
de^{i\theta_0}(T(x))=T(e^{i\theta_0}\circ x).
\end{split}
\end{equation}
Let $(de^{i\theta_0})^\ast: \Lambda^q(\mathbb CT^\ast X)\rightarrow\Lambda^q(\mathbb CT^\ast X)$ be the pull back of $de^{i\theta_0}, q=0,1\cdots, n-1$. From (\ref{a}), we can check that for every $q=0, 1,\cdots, n-1$
\begin{equation}\label{j1}
(de^{i\theta_0})^\ast: T^{\ast0,q}_{e^{i\theta_0}\circ x}X\rightarrow T_x^{\ast0,q}X.
\end{equation}

Let $u\in\Omega^{0,q}(X)$. Define $Tu$ as follows. For any $X_1,\cdots, X_q\in T_x^{1,0}X$,
\begin{equation}\label{b}
Tu(X_1,\cdots, X_q):=\frac{\partial}{\partial\theta}\left((de^{i\theta})^\ast u(X_1,\cdots, X_q)\right)\Big|_{\theta=0}.
\end{equation}
From (\ref{j1}) and (\ref{b}), we have that $Tu\in\Omega^{0, q}(X)$ for all $u\in\Omega^{0, q}(X)$. See the discussion before Lemma~\ref{b5} for another way to define $Tu$.

Let $\overline\partial_b:\Omega^{0,q}(X)\rightarrow\Omega^{0,q+1}(X)$ be the tangential Cauchy-Riemann operator. It is straightforward from (\ref{a}) and (\ref{b}) to see that
\begin{equation}\label{c}
T\overline\partial_b=\overline\partial_bT~\text{on}~\Omega^{0,q}(X)
\end{equation}
(see also \eqref{e-gue150620}).
For every $m\in\mathbb Z$, put $\Omega^{0,q}_m(X):=\{u\in\Omega^{0,q}(X): Tu=imu\}$.
From (\ref{c}) we have the $\ddbar_b$-complex for every $m\in\mathbb Z$:
\begin{equation}\label{e-gue140903VI}
\ddbar_b:\cdots\To\Omega^{0,q-1}_m(X)\To\Omega^{0,q}_m(X)\To\Omega^{0,q+1}_m(X)\To\cdots.
\end{equation}
For every $m\in\mathbb Z$, the $q$-th $\ddbar_b$ cohomology (or Kohn-Rossi cohomology) is given by
\begin{equation}\label{a8}
H^{q}_{b,m}(X):=\frac{{\rm Ker\,}\ddbar_{b}:\Omega^{0,q}_m(X)\To\Omega^{0,q+1}_m(X)}{\operatorname{Im}\ddbar_{b}:\Omega^{0,q-1}_m(X)\To\Omega^{0,q}_m(X)}.
\end{equation}
The starting point of this paper is that without any Levi curvature assumption, for every $m\in\mathbb Z$ and every
$q=0,1,2,\ldots,n-1$ we have
\begin{equation}\label{a9I}
{\rm dim\,}H^{q}_{b,m}(X)<\infty.
\end{equation}

\begin{definition}
We say that a function $u\in C^\infty(X)$ is a Cauchy-Riemann (CR for short) function
if $\overline\partial_bu=0$ or in the other word, $\overline Zu=0$ for all $Z\in C^{\infty}(X, T^{1, 0}X)$.
\end{definition}

For $m\in\mathbb Z$, when $q=0$, $H^0_{b, m}(X)$ is the space of CR functions which lie in the eigenspace of $T$ and we call $H^0_{b, m}(X)$ the $m$-th Fourier component of CR functions.

\subsection{Hermitian CR geometry}

We need

\begin{definition}
Let $D$ be an open set and let $V\in C^\infty(D, \mathbb CTX)$ be a vector field on $D$. We say that $V$ is $T$-rigid if
\begin{equation}
de^{i\theta_0}(V(x))=V(e^{i\theta_0}\circ x)
\end{equation}
for any $x, \theta_0\in[0,2\pi)$ satisfying $x\in D, e^{i\theta_0}\circ x\in D.$
\end{definition}

\begin{definition}
Let $\langle\cdot|\cdot\rangle$ be a Hermitian metric on $\mathbb CTX$.
We say that $\langle\cdot|\cdot\rangle$ is $T$-rigid if for $T$-rigid vector fields $V, W$ on $D$, where $D$ is any open set, we have
\begin{equation}
\langle V(x)|W(x)\rangle=\langle (de^{i\theta_0}V)(e^{i\theta_0}\circ x)|(de^{i\theta_0}W)(e^{i\theta_0}\circ x)\rangle, \forall x\in D, \theta_0\in[0,2\pi).
\end{equation}
\end{definition}

\begin{lemma}[Theorem 9.2 in \cite{H14}]\label{a7}
Let $X$ be a compact connected CR manifold with a transversal CR $S^1$-action. There is always a $T$-rigid Hermitian metric $\langle\cdot|\cdot\rangle$ on $\mathbb CTX$ such that $T^{1,0}X\bot T^{0,1}X, T\bot(T^{1,0}X\oplus T^{0,1}X), \langle T|T\rangle=1$ and $\langle u|v\rangle$ is real if $u, v$ are real tangent vectors.
\end{lemma}

From now on, we fix a $T$-rigid Hermitian metric $\langle\cdot|\cdot\rangle$ on $\mathbb CTX$ satisfying all the properties in Lemma~\ref{a7}. The Hermitian metric $\langle\cdot|\cdot\rangle$ on
$\mathbb CTX$ induces by duality a Hermitian metric on $\mathbb CT^\ast X$ and also on the bundles of $(0,q)$-forms $T^{\ast 0,q}X, q=0,1\cdots,n-1.$ We shall also denote all these induced
metrics by $\langle\cdot|\cdot\rangle$. For every $v\in T^{\ast0,q}X$, we write
$|v|^2:=\langle v|v\rangle$. We have the pointwise orthogonal decompositions:
\begin{equation}
\begin{split}
&\mathbb CT^{\ast}X=T^{\ast1,0}X\oplus T^{\ast0,1}X\oplus\{\lambda\omega_0:\lambda\in\mathbb C\},\\
&\mathbb CTX=T^{1,0}X\oplus T^{0,1}X\oplus\{\lambda T:\lambda\in\mathbb C\}.
\end{split}
\end{equation}

For any $p\in X$, locally there is an orthonormal frame $\{U_1, \ldots, U_{n-1}\}$ of $T^{1, 0}X$ with respect to the given $T$-rigid Hermitian metric $\langle\cdot|\cdot\rangle$ such that the Levi-form $\mathcal L_p$ is diagonal in this frame, $\mathcal L_p(U_i, \overline{U_j})=\lambda_j\delta_{ij}$, where $\delta_{ij}=1$ if $i=j$, $\delta_{ij}=0$ if $i\neq j$. The entries $\{\lambda_1, \ldots, \lambda_{n-1}\}$ are called the eigenvalues of Levi-form at $p$ with respect to the $T$-rigid Hermitian metric $\langle\cdot|\cdot\rangle$. Moreover, the determinant of $\mathcal L_p$ is defined by $\det\mathcal L_p=\lambda_1(p)\cdots\lambda_{n-1}(p)$.

Let $(\,\cdot\,|\,\cdot\,)$ be the $L^2$ inner product on $\Omega^{0,q}(X)$ induced by $\langle\,\cdot\,|\,\cdot\,\rangle$ and let $\norm{\cdot}$ denote the corresponding norm. Then for all $u, v\in\Omega^{0,q}(X)$
\begin{equation}
(u|v)=\int_X\langle u| v\rangle dv_X,
\end{equation}
where $dv_X$ is the volume form on $X$ induced by the $T$-rigid Hermitian metric.
As before, for $m\in\mathbb Z$, we denote by
\begin{equation}
\Omega_m^{0,q}(X)=\{u\in\Omega^{0,q}(X): Tu=imu\}
\end{equation}
the eigenspace of $T$.
Let $L^2_{(0,q),m}(X)$ be the completion of $\Omega_m^{0,q}(X)$ with respect to $(\cdot|\cdot)$. For $m\in\mathbb Z$, let
\begin{equation}
Q^{(q)}_{m}: L^2_{(0,q)}(X)\rightarrow L^2_{(0,q),m}(X)
\end{equation}
be the orthogonal projection with respect to $(\cdot|\cdot)$. Then for any $u\in \Omega^{0, q}(X)$ $$Q_m^{(q)}u=\frac{1}{2\pi}\int_{-\pi}^{\pi}u(e^{i\theta}\circ x)e^{-im\theta}d\theta.$$ By using the elementary Fourier analysis, it is straightforward to see that for any $u\in\Omega^{0, q}(X)$,
\begin{equation}
\sum_{m=-N}^NQ^{(q)}_mu\rightarrow u~\text{in}~C^\infty~\text{topology as}~N \rightarrow\infty.
\end{equation}
Thus for every $u\in L^2_{(0, q)}(X)$,
\begin{equation}
\sum_{m=-N}^NQ^{(q)}_mu\rightarrow u~\text{in}~L^2_{(0, q)}(X, L^k)~\text{as}~N \rightarrow\infty.
\end{equation}
If we denote the  $\lim\limits_{N\rightarrow\infty}\sum\limits_{m=-N}^NQ^{(q)}_mu$ by $\sum\limits_{m\in\mathbb Z}Q^{(q)}_mu$, then we write $u=\sum\limits_{m\in\mathbb Z}Q^{(q)}_mu$. Thus, we have the following Fourier decomposition:
\begin{equation}\label{x5}
\begin{split}
\Omega^{0, q}(X)=\bigoplus\limits_{m\in\mathbb Z} \Omega^{0, q}_m(X),
L^2_{(0, q)}(X)=\bigoplus\limits_{m\in\mathbb Z} L^2_{(0, q), m}(X).
\end{split}
\end{equation}
By (\ref{a8}) and (\ref{x5}), we have the following Fourier decomposition of the $q$-th Kohn-Rossi cohomology (see (1.39) in \cite{MM06})
\begin{equation}\label{x6}
H^q_{b}(X)\cong\bigoplus\limits_{m\in\mathbb Z} H^q_{b, m}(X).
\end{equation}

Let $\overline\partial_{b}^\ast: \Omega^{0,q+1}(X)\rightarrow\Omega^{0,q}(X)$ be the formal adjoint of $\overline\partial_b$ with respect to $(\cdot|\cdot)$. Since the Hermitian metrics $\langle\cdot|\cdot\rangle$ are $T$-rigid, we can check that
\begin{equation}\label{d}
T\overline\partial_{b}^\ast=\overline\partial_{b}^\ast T~\text{on}~\Omega^{0,q}(X), \forall q=1,\cdots,n-1
\end{equation}
and from (\ref{d}) we have
\begin{equation}\label{e}
\overline\partial_{b}^\ast:\Omega_m^{0,q+1}(X)\rightarrow\Omega_m^{0,q}(X),\forall m\in\mathbb Z.
\end{equation}
Put
$$\Box_{b}^{(q)}:=\overline\partial_b\overline\partial_{b}^\ast
+\overline\partial_{b}^\ast\overline\partial_b:\Omega^{0,q}(X)\rightarrow\Omega^{0,q}(X).$$
Combining (\ref{c}), (\ref{d}) and (\ref{e}), we have
\begin{equation}\label{a9}
T\Box_{b}^{(q)}=\Box^{(q)}_{b}T~\text{on}~\Omega^{0,q}(X), \forall q=0,1,\cdots,n-1
\end{equation}
and (\ref{a9}) implies that
\begin{equation}
\Box_{b}^{(q)}:\Omega_m^{0,q}(X)\rightarrow\Omega_m^{0,q}(X),\forall m\in\mathbb Z.
\end{equation}
We will write $\Box^{(q)}_{b, m}$ to denote the restriction of $\Box^{(q)}_{b}$ on $\Omega_m^{0,q}(X)$. For every $m\in\mathbb Z$, we extend $\Box^{(q)}_{b, m}$ to $L^2_{(0,q),m}(X)$ by
\begin{equation}\label{j2}
\Box^{(q)}_{b, m}:{\rm Dom}(\Box^{(q)}_{b,m})\subset L^2_{(0,q),m}(X)\rightarrow L^2_{(0,q),m}(X),
\end{equation}
where ${\rm Dom}(\Box^{(q)}_{b,m})=\{u\in L^2_{(0,q),m}(X): \Box^{(q)}_{b, m}u\in L^2_{(0,q),m}(X)~~\text{in the sense of distribution}\}$.

The following result follows from Kohn's $L^2$-estimate (see Theorem 8.4.2 in \cite{CS01}).

\begin{theorem}\label{f}
For every $s\in\mathbb{N}_0:=\mathbb N\cup\{0\}$, there exists a constant $C_{s}>0$ such that
\begin{equation}
\|u\|_{s+1}\leq C_{s}\left(\|\Box^{(q)}_{b}u\|_{s}+\|Tu\|_s+\|u\|_s\right),\forall u\in\Omega^{0,q}(X)
\end{equation}
where $\|\cdot\|_s$ denotes the standard Sobolev norm of order $s$ on $X$.
\end{theorem}

From Theorem \ref{f}, we deduce that

\begin{theorem}\label{g}
Fix $m\in\mathbb Z$, for every $s\in\mathbb N_0$, there is a constant $C_{s, m}>0$ such that
\begin{equation}\label{xx}
\|u\|_{s+1}\leq C_{s, m}\left(\|\Box^{(q)}_{b, m}u\|_s+\|u\|_s\right), \forall u\in\Omega^{0,q}_m(X).
\end{equation}
\end{theorem}

From Theorem~\ref{g} and some standard argument in functional analysis, we deduce the following Hodge theory for $\Box^{(q)}_{b, m}$ (see Section 3 in~\cite{CHT15})

\begin{theorem}\label{gI}
Fix $m\in\mathbb Z$. $\Box^{(q)}_{b, m}:\mathrm{Dom}(\Box^{(q)}_{b, m})\subset L^2_{(0,q), m}(X)\rightarrow L^2_{(0,q), m}(X)$ is a self-adjoint operator. The spectrum of $\Box^{(q)}_{b, m}$ denoted by
$\mathrm {Spec}(\Box^{(q)}_{b, m})$ is a discrete subset of $[0,\infty)$. For every $\lambda\in\mathrm{Spec}(\Box^{(q)}_{b, m})$ the eigenspace with respect to $\lambda$
\begin{equation}\label{h}
\mathcal{H}^q_{b, m,\lambda}(X)=\left\{u\in\mathrm{Dom}(\Box^{(q)}_{b, m}):\Box^{(q)}_{b, m}u=\lambda u\right\}
\end{equation}
is finite dimensional with $\mathcal{H}^q_{b, m,\lambda}(X)\subset\Omega^{0,q}_m(X)$ and for $\lambda=0$ we denote by $\mathcal H^q_{b, m}(X)$ the harmonic space $\mathcal H^q_{b, m,0}(X, L^k)$ for brevity and then we have the Dolbeault isomorphism
\begin{equation}\label{i}
\mathcal{H}^q_{b, m}(X)\cong H^q_{b,m}(X).
\end{equation}
In particular, from (\ref{i}) we have
\begin{equation}\label{j3}
\dim H^q_{b,m}(X)<\infty, \forall ~m\in\mathbb Z, \forall ~0\leq q\leq n-1.
\end{equation}
\end{theorem}

\begin{remark}
We would like to mention that transversal property (\ref{x2}) of the $S^1$-action  is a necessary condition for the finite dimension of ${\rm dim}H^q_{b, m}(X)$. In fact, we have the following counterexample when the $S^1$-action is not transversal. Let $$X=\mathbb S^3:=\{(z_1, z_2)\in\mathbb C^2: |z_1|^2+|z_2|^2=1\}.$$ The $S^1$-action on $X$ is defined by $e^{i\theta}\circ(z_1, z_2)=(e^{i\theta} z_1, e^{-in\theta}z_2), n\geq 1$. Let $T$ be the global induced vector field. By definition
$$T=i\left(z_1\frac{\partial}{\partial z_1}-\overline z_1\frac{\partial}{\partial\overline z_1}-nz_2\frac{\partial}{\partial z_2}+n\overline z_2\frac{\partial}{\partial\overline z_2}\right).$$
We can check that $T$ is not transversal to $T^{1,0}X\bigoplus T^{0, 1}X$ and thus the $S^1$-action defined above is not transversal. Moreover, we have ${\rm dim}H^0_{b, m}(X)=\infty$. This is because the functions $\{u_k=z_1^{m+nk}z_2^k, k=1,2,3,\ldots\}\subset H^0_{b, m}(X)$ when restricted on $X$.
\end{remark}

\subsection{Canonical local coordinates}\label{q}

In this work, we need the following result due to Baouendi-Rothschild-Treves, (see Proposition I.2. in \cite{BRT85}).

\begin{theorem}\label{j}
Let $X$ be a compact CR manifold of ${\rm dim}X=2n-1, n\geq2$ with a transversal CR $S^1$- action. Let $\langle\cdot|\cdot\rangle$ be the given $T$-rigid Hermitian metric on $X$.
For every point $x_0\in X$, there exists local coordinates $(x_1,\cdots,x_{2n-1})=(z,\theta)=(z_1,\cdots,z_{n-1},\theta), z_j=x_{2j-1}+ix_{2j},j=1,\cdots,n-1, x_{2n-1}=\theta$, defined in some small neighborhood $D=\{(z, \theta): |z|<\varepsilon, |\theta|<\delta\}$ of $x_0$ such that
\begin{equation}\label{e-can}
\begin{split}
&T=\frac{\partial}{\partial\theta}\\
&Z_j=\frac{\partial}{\partial z_j}+i\frac{\partial\varphi(z)}{\partial z_j}\frac{\partial}{\partial\theta},j=1,\cdots,n-1
\end{split}
\end{equation}
where $\{Z_j(x)\}_{j=1}^{n-1}$ form a basis of $T_x^{1,0}X$, for each $x\in D$ and $\varphi(z)\in C^\infty(D,\mathbb R)$ is independent of $\theta$. Moreover, on $D$ we can take $(z,\theta)$ and $\varphi$ so that $(z(x_0),\theta(x_0))=(0,0)$ and $\varphi(z)=\sum\limits_{j=1}^{n-1}\lambda_j|z_j|^2+O(|z|^3), \forall (z, \theta)\in D$, where $\{\lambda_j\}_{j=1}^{n-1}$ are the eigenvalues of Levi-form of $X$ at $x_0$ with respect to the given $T$-rigid Hermitian metric on $X$.
\end{theorem}

\begin{remark}
Let $D$ be as in Theorem~\ref{j}. We will always identify $D$ with an open set of $\mathbb R^{2n-1}$ and
we call $D$ canonical local patch and $(z, \theta,\varphi)$ canonical coordinates. The constants $\varepsilon$ and $\delta$ in Theorem \ref{j} depend on $x_0$. Let $x_0\in D$. We say that $(z,\theta,\varphi)$ is trivial at $x_0$ if $(z(x_0),\theta(x_0))=(0,0)$ and $\varphi(z)=\sum\limits_{j=1}^{n-1}\lambda_j|z_j|^2+O(|z|^3)$, where $\{\lambda_j\}_{j=1}^{n-1}$ are the eigenvalues of Levi-form of $X$ at $x_0$ with respect to the $T$-rigid Hermitian metric $\langle\,\cdot\,|\,\cdot\,\rangle$.
\end{remark}

\begin{lemma}\label{l-gue150615}
Let $x_0\in X_{\rm reg}$. Then we can find canonical coordinates $(z,\theta,\varphi)$ defined in $D=\{(z,\theta): |z|<\varepsilon_0, |\theta|<\pi\}$ such that $(z,\theta,\varphi)$ is trivial at $x_0$.
\end{lemma}

\begin{proof}
Let $(z,\theta,\varphi)$ be any canonical coordinates defined in $D_1=\{(z,\theta): |z|<\varepsilon_1, |\theta|<\delta\}$ such that $(z,\theta,\varphi)$ is trivial at $x_0$. We identify $D_1$ with an open neighborhood of $x_0$. It is clear that
\begin{equation}\label{e-gue150615I}
e^{it}\circ (z_1,0)\neq (z_2,0),\ \ \forall 0<|t|<\delta, |z_1|<\varepsilon_1, |z_2|<\varepsilon_1.
\end{equation}
We claim that
\begin{equation}\label{e-gue150615II}
\mbox{There is a $0<\varepsilon_0<\varepsilon_1$ such that $e^{it}\circ (z_1,0)\neq (z_2,0)$, $\forall\frac{\delta}{2}\leq t\leq2\pi-\frac{\delta}{2}, |z_1|<\varepsilon_0, |z_2|<\varepsilon_0$.}
\end{equation}
If the claim is not true, for every $j\in\mathbb N$, we can find $z_{1,j}, z_{2,j}\in\mathbb C^{n-1}$, $\theta_j\in\mathbb R$ with $|z_{1,j}|<\frac{\varepsilon_1}{j}$, $|z_{2,j}|<\frac{\varepsilon_1}{j}$, $\frac{\delta}{2}\leq\theta_j\leq2\pi-\frac{\delta}{2}$ such that
\begin{equation}\label{e-gue150615III}
e^{i\theta_j}\circ(z_{1,j},0)=(z_{2,j},0),\ \ j=1,2,\ldots.
\end{equation}
From \eqref{e-gue150615III}, we get $e^{i\theta_0}\circ(0,0)=(0,0)$, for some $\frac{\delta}{2}\leq\theta_0\leq2\pi-\frac{\delta}{2}$. But $x_0\in X_{\rm reg}$, we get a contradiction. The claim follows.

Let $0<\varepsilon_0<\varepsilon_1$ be as in \eqref{e-gue150615II}. Consider the map
\[\begin{split}
\Phi:\{z\in\mathbb C^{n-1}: |z|<\varepsilon_0\}\times\{\theta\in\mathbb R: |\theta|<\pi\}&\longmapsto X,\\
(z,\theta)&\longmapsto e^{i\theta}\circ(z,0).
\end{split}\]
We claim that $\Phi$ is injective. If $e^{i\theta_1}\circ(z_1,0)=e^{i\theta_2}\circ(z_2,0)$, for some $|z_1|<\varepsilon_0$, $|z_2|<\varepsilon_0$, $|\theta_1|<\pi$, $|\theta_2|<\pi$. We have $e^{i(\theta_1-\theta_2)}\circ(z_1,0)=(z_2,0)$. We may assume that $\theta_1\geq\theta_2$. From \eqref{e-gue150615II}, we see that $0\leq\theta_1-\theta_2\leq\frac{\delta}{2}$ or $2\pi-\frac{\delta}{2}\leq\theta_1-\theta_2<2\pi$. If $2\pi-\frac{\delta}{2}\leq\theta_1-\theta_2<2\pi$. Then, $-\frac{\delta}{2}\leq\theta_1-\theta_2-2\pi<0$ and $e^{i(\theta_1-\theta_2-2\pi)}\circ(z_1,0)=(z_2,0)$. By \eqref{e-gue150615I}, we get a contradiction. We must have $0\leq\theta_1-\theta_2\leq\frac{\delta}{2}$. From \eqref{e-gue150615II}, we deduce that $\theta_1=\theta_2$ and $z_1=z_2$. Thus, $\Phi$ is injective. When $|z|<\varepsilon_0$, we can extend $\theta$ to $|\theta|<\pi$ by $\Phi$. The lemma follows.
\end{proof}

In the proof of Theorem~\ref{m}, we need the following

\begin{lemma}\label{l-gue150616}
Let $x_0\in X_k$, $k\in\mathbb N$, $k>1$. For every $\epsilon>0$, $\epsilon$ small, we can find canonical coordinates $(z,\theta,\varphi)$ defined in  $D_\epsilon=\{(z,\theta): |z|<\varepsilon_0, |\theta|<\frac{\pi}{k}-\epsilon\}$ such that $(z,\theta,\varphi)$ is trivial at $x_0$.
\end{lemma}

\begin{proof}
Let $(z,\theta,\varphi)$ be any canonical coordinates defined in $D_1=\{(z,\theta): |z|<\varepsilon_1, |\theta|<\delta\}$ such that $(z,\theta,\varphi)$ is trivial at $x_0$. We identify $D_1$ with an open neighborhood of $x_0$. It is clear that
\begin{equation}\label{e-gue150616}
e^{it}\circ (z_1,0)\neq (z_2,0),\ \ \forall 0<|t|<\delta, |z_1|<\varepsilon_1, |z_2|<\varepsilon_1.
\end{equation}
Fix $\epsilon>0$, $\epsilon$ small. We claim that
\begin{equation}\label{e-gue150616I}
\mbox{There is a $0<\varepsilon_0<\varepsilon_1$ such that $e^{it}\circ (z_1,0)\neq (z_2,0)$, $\forall\frac{\delta}{2}\leq t\leq\frac{2\pi}{k}-\frac{\epsilon}{2}, |z_1|<\varepsilon_0, |z_2|<\varepsilon_0$.}
\end{equation}
If the claim is not true, for every $j\in\mathbb N$, we can find $z_{1,j}, z_{2,j}\in\mathbb C^{n-1}$, $\theta_j\in\mathbb R$ with $|z_{1,j}|<\frac{\varepsilon_1}{j}$, $|z_{2,j}|<\frac{\varepsilon_1}{j}$, $\frac{\delta}{2}\leq \theta_j\leq\frac{2\pi}{k}-\frac{\epsilon}{2}$ such that
\begin{equation}\label{e-gue150616II}
e^{i\theta_j}\circ(z_{1,j},0)=(z_{2,j},0),\ \ j=1,2,\ldots.
\end{equation}
From \eqref{e-gue150616II}, we get $e^{i\theta_0}\circ(0,0)=(0,0)$, for some $\frac{\delta}{2}\leq\theta_0\leq\frac{2\pi}{k}-\frac{\epsilon}{2}$. But $x_0\in X_k$, we get a contradiction. The claim follows.

Let $0<\varepsilon_0<\varepsilon_1$ be as in \eqref{e-gue150616I}. Consider the map
\[\begin{split}
\Phi_\epsilon:\{z\in\mathbb C^{n-1}: |z|<\varepsilon_0\}\times\{\theta\in\mathbb R: |\theta|<\frac{\pi}{k}-\epsilon\}&\longmapsto X,\\
(z,\theta)&\longmapsto e^{i\theta}\circ(z,0).
\end{split}\]
We claim that $\Phi_\epsilon$ is injective. If $e^{i\theta_1}\circ(z_1,0)=e^{i\theta_2}\circ(z_2,0)$, for some $|z_1|<\varepsilon_0$, $|z_2|<\varepsilon_0$, $|\theta_1|<\frac{\pi}{k}-\epsilon$, $|\theta_2|<\frac{\pi}{k}-\epsilon$. We have $e^{i(\theta_1-\theta_2)}\circ(z_1,0)=(z_2,0)$. We may assume that $\theta_1\geq\theta_2$ and hence $0\leq\theta_1-\theta_2<\frac{2\pi}{k}-2\epsilon$. From \eqref{e-gue150616I}, we see that $0\leq\theta_1-\theta_2\leq\frac{\delta}{2}$. From \eqref{e-gue150616}, we deduce that $\theta_1=\theta_2$ and $z_1=z_2$. Thus, $\Phi_\epsilon$ is injective. When $|z|<\varepsilon_0$, we can extend $\theta$ to $|\theta|<\frac{\pi}{k}-\epsilon$ by $\Phi_\epsilon$. The lemma follows.
\end{proof}

By using canonical coordinates, we get another way to define $Tu, \forall u\in\Omega^{0,q}(X)$. Let $D$ be a canonical local patch with canonical coordinates $(z,\theta,\varphi)$. By (\ref{e-can}), $\{dz_j\}_{j=1}^{n-1}$ is the dual frame of $\{Z_j\}_{j=1}^{n-1}$. For a multi-index $J=(j_1,\ldots,j_q)\in\{1,2,\ldots,n\}^q$ we set $|J|=q$. We say that $J$ is strictly increasing if $1\leq j_1<j_2<\cdots<j_q\leq n$. We put $d\overline z^J=d\overline z_{j_1}\wedge d\overline z_{j_2}\wedge\cdots\wedge d\overline z_{j_q}$. It is clearly that $\{d\overline z^J, |J|=q, J~\text{strictly increasing}\}$ is a basis for $T^{\ast0,q}_xX$ for every $x\in D$. Let $u\in\Omega^{0,q}(X)$. On $D$, we write
$u=\sum_{|J|=q}^\prime u_J d\overline z^J$, where the notation $\sum^\prime$ means the summation over strictly increasing multiindices.
Then on $D$ we can check that
\begin{equation}\label{e-gue150620}
Tu=\sideset{}{'}\sum_{|J|=q} Tu_Jd\overline z^J,\ \ \ddbar_bTu=T\ddbar_bu.
\end{equation}

\begin{lemma}\label{b5}
Fix $x_0\in X$ and let $D=\tilde D\times(-\delta,\delta)\subset\mathbb C^{n-1}\times\mathbb R$ be a canonical local patch with canonical coordinates $(z,\theta,\varphi)$ such that  $(z,\theta,\varphi)$ is trivial at $x_0$.
We can find orthonormal frame $\{e^j\}_{j=1}^{n-1}$ of $T^{\ast0,1}X$ with respect to the fixed $T$-rigid Hermitian metric such that on $D=\tilde D\times(-\delta, \delta)$, we have $e^j(x)=e^j(z)=d\overline z_j+O(|z|), \forall x=(z, \theta)\in D,  j=1,\cdots,n-1$. Moreover, if we denote by $dv_X$ the volume form with respect to the $T$-rigid Hermitian metric on $\mathbb CTX$, then on $D$ we have
$dv_X=\lambda(z)dv(z)d\theta$
with $\lambda(z)\in C^\infty(\tilde D, \mathbb R)$ which does not depend on $\theta$ and $dv(z)=2^{n-1}dx_1\cdots dx_{2n-2}$.
\end{lemma}

\begin{proof}
From the definition of the $T$-rigid Hermitian metric, we can check that the inner
product $\langle dz_k|dz_j\rangle$ does not depend on $\theta$. We denote by $g^{\overline k j}(z)=\langle d\overline z_k|d\overline z_j\rangle$ on $D$. Taking coordinate transformation of $z=(z_1, \ldots, z_{n-1})$ if needed such that $g^{\overline kj}(x_0)=\delta_{kj}$. By Gram-Schmidt process, we can find an orthonormal frame $\{e^j\}_{j=1}^{n-1}$ of $T^{\ast 0, 1}X$.
Write $e^j=\sum_{k=1}^{n-1}b_{j\overline k}d\overline z_k, j=1,
\ldots, n-1$. Since $g^{\overline kj}(z)=\langle d\overline z_k|d\overline z_j\rangle$ does not depend on $\theta$ on $D$, we can check that coefficients $\{b_{j\overline k}\}_{1\leq j, k\leq n-1}$ do not depend on $\theta$. Then $e^j(x)=d\overline z_j+O(z)$. Since $-\omega_0(z,
\theta)=d\theta+\sum_{j=1}^{n-1}(\alpha_j(z,
\theta)dz_j+\overline{\alpha_j(z, \theta)}d\overline z_j)$ and $\{e^1,
\ldots, e^{n-1}, \omega_0\}$ is an orthonormal frame of $\mathbb CTX$
over $D$, then the volume form on $D$ is defined by
$dv_X=\sqrt{-1}^{n-1}\overline e^1\wedge e^1\wedge\cdots\wedge\overline
e^{n-1}\wedge e^{n-1}\wedge(-\omega_0)$. The lemma follows.
\end{proof}

\begin{remark}\label{g1}
For any $x_0\in X$, let $D=\tilde D\times (-\delta, \delta)$ be a
canonical local patch with canonical coordinates $(z, \theta,\varphi)$ such that $(z, \theta,\varphi)$ is trivial at $x_0$. Here, $\tilde D=\{z\in\mathbb C^{n-1}: |z|<\varepsilon\}$. We identify $\tilde D$ with an open subset of $\mathbb C^{n-1}$ with complex coordinates $z=(z_1, \ldots, z_{n-1}).$ Since
$\{d\overline z_j\}_{j=1}^{n-1}$ is a frame of $T^{\ast 0, 1}X$ over
$D$, we will treat them as the frame of $T^{\ast 0, 1}\tilde D$ which is
the bundle of $(0, 1)$-forms over the domain $\tilde D.$ Let
$(g^{\overline k j}(z))$ be the Hermitian matrix defined in the proof of
Lemma~\ref{b5}. Then we define a Hermitian metric on $T^{\ast 0, 1}\tilde D$ given by $(g^{\overline k j}(z))_{j, k=1}^{n-1}$ with $\langle d\overline z_k|d\overline z_j\rangle=g^{\overline k j}$. We also denote by
$\langle\cdot|\cdot\rangle$ the Hermitian metric on $T^{\ast 0, 1}\tilde D$. By duality, it will induces a Hermitian metric on $T^{ 0, 1}\tilde D$. We extend the Hermitian metric to  $\mathbb CT\tilde D$ and $T^{\ast 0, q}\tilde D$ in the standard way and denote all the Hermitian metrics by $\langle\cdot|\cdot\rangle$. Then $\{e^j(z)\}_{j=1}^{n-1}$ defined in Lemma~\ref{b5} is also an
orthonormal frame of $T^{\ast 0, 1}\tilde D$. With respect to the given  Hermitian metric on $T^{\ast 0, 1}\tilde D$, the volume form on $\tilde
D$ is given by $\lambda(z)dv(z)$. Here, $\lambda(z)\in C^\infty(\tilde D, \mathbb R)$ is the
function defined in Lemma~\ref{b5}.
\end{remark}

\subsection{The scaling technique}\label{kk}

In this section, we will recall the scaling technique in \cite{Be04}, developed in \cite{HM12}, \cite{HM14}, \cite{H15} and \cite{HL15}.
Fix  $x_0\in X$, we take canonical local patch $D=\tilde D\times (-\delta, \delta)=\{(z, \theta): |z|<\varepsilon, |\theta|<\delta\}$ with canonical coordinates $(z,\theta,\varphi)$ such that $(z,\theta,\varphi)$ is trivial at $x_0$. In this section, we identify $\tilde D$ with an open subset of $\mathbb C^{n-1}=\mathbb R^{2n-2}$ with complex coordinates $z=(z_1, \ldots, z_{n-1})$. Let $L_1\in T^{1,0}\tilde D,\cdots, L_{n-1}\in T^{1,0}\tilde D$ be the dual frame of $\overline{e^1},\cdots, \overline{e^{n-1}}$ with respect to the Hermitian metric $\langle\cdot|\cdot\rangle$ defined in Remark~\ref{g1}. The Hermitian metric $\langle\cdot|\cdot\rangle$ on $\tilde D$ we have chosen in Lemma~\ref{b5} and Remark~\ref{g1} implies that
\begin{equation}\label{e-gue150303}
\left\langle\frac{\partial}{\partial x_j}(0)\Big|\frac{\partial}{\partial x_t}(0)\right\rangle=2\delta_{jt}
\end{equation}
for $j,t=1,\ldots, 2n-2$, and in the  coordinates $z=(z_1,\ldots, z_{n-1})$,
$
L_j=\frac{\partial}{\partial z_j}+O(z), j=1,\ldots,n-1,
$
where $\frac{\partial}{\partial z_j}=\frac{1}{2}(\frac{\partial}{\partial{x_{2j-1}}}-i\frac{\partial}{\partial x_{2j}}), j=1,\ldots, n-1$.

Let $M\subset\mathbb C^{n-1}$ be an open set.
Let $\Omega^{0,q}(M)$ be the space of smooth $(0,q)$-forms on $M$ and let $\Omega^{0,q}_0(M)$ be the subspace of $\Omega^{0,q}(M)$ whose elements have compact support in $M$.
Let $(\cdot|\cdot)_{2\varphi}$ be the  weighted inner product on the space $\Omega^{0,q}_0(\tilde D)$ defined as follows:
\begin{equation}
(f|g)=\int_{\tilde D}\langle f|g\rangle e^{-2\varphi(z)}\lambda(z)dv(z)
\end{equation}
where $f, g\in\Omega_0^{0,q}(\tilde D)$ and $\lambda(z)$ is as in Lemma~\ref{b5}. We denote by $L^2_{(0,q)}(\tilde D, 2\varphi)$  the completion of $\Omega_0^{0,q}(\tilde D)$ with respect to $(\cdot|\cdot)_{2\varphi}$.  For $r>0$, let
$\tilde D_r=\{z\in\mathbb C^{n-1}: |z|<r\}$. Here $\{z\in\mathbb
C^{n-1}:|z|<r\}$ means that $\{z\in\mathbb C^{n-1}:|z_j|<r, j=1,\cdots, n-1\}$. For $m\in\mathbb N$, let $F_m$ be the scaling map $F_m(z)=(\frac{z_1}{\sqrt m},\ldots, \frac{z_{n-1}}{\sqrt m}), z\in \tilde D_{\log m}$. From now on, we
assume $m$ is sufficiently large such that $F_m(\tilde D_{\log m})\Subset \tilde D$. We define the scaled bundle $F_m^\ast T^{\ast0, q} \tilde D$ on $\tilde D_{\log m}$ to be the bundle whose fiber at
$z\in \tilde D_{\log m}$ is
\begin{equation}
F_m^\ast T^{\ast0, q}\tilde D|_{z}=\left\{\sum\nolimits_{|J|=q}^\prime a_Je^J\left(\frac{z}{\sqrt m}\right): a_J\in\mathbb C, |J|=q, J~\text{strictly increasing}\right\}.
\end{equation}
We take the Hermitian metric $\langle\cdot|\cdot\rangle_{F_m^\ast}$ on $F_m^\ast T^{\ast 0, q} \tilde D$ so that at each point $z\in \tilde D_{\log m}$,
\begin{equation}
\left\{e^J\left(\frac{z}{\sqrt m}\right): |J|=q, J~\text{strictly increasing}\right\}
\end{equation}
is an orthonormal frame for $F_m^\ast T^{\ast0, q} \tilde D$ on $\tilde D_{\log m}$.

Let $F_m^\ast\Omega^{0,q}(\tilde D_r)$ denote the space of smooth sections of $F_m^\ast T^{\ast0, q} \tilde D$ over $\tilde D_r$ and let $F_m^\ast\Omega^{0,q}_0(\tilde D_r)$ be the subspace of $F_m^\ast\Omega^{0,q}(\tilde D_r)$ whose elements have compact support in $\tilde D_r$. Given $f\in\Omega^{0,q}(\tilde D_r)$.
We write $f=\sum\nolimits_{|J|=q}^\prime f_Je^J$. We define the scaled form $F_m^\ast f\in F_m^\ast\Omega^{0,q}(\tilde D_{\log m})$ by
\begin{equation}\label{j8}
F_m^\ast f=\sum\nolimits_{|J|=q}^\prime f_J\left(\frac{z}{\sqrt m}\right)e^J\left(\frac{z}{\sqrt m}\right), z\in\tilde D_{\log m}.
\end{equation}
For brevity, we denote $F_m^\ast f$ by $f(\frac{z}{\sqrt m})$.
Let $P$ be a partial differential operator of order one on $F_m(\tilde D_{\log m})$ with $C^\infty$ coefficients. We write $P=\sum\limits_{j=1}^{2n-2}a_j(z)\frac{\partial}{\partial x_j}.$ The scaled partial differential operator $P_{(m)}$ on $\tilde D_{\log m}$
is given by
$
P_{(m)}=\sum_{j=1}^{2n-2}F_m^\ast a_j\frac{\partial}{\partial x_j}.
$
Let $f\in C^\infty(F_m(\tilde D_{\log m}))$. We can check that
\begin{equation}\label{j9}
P_{(m)}(F_m^\ast f)=\frac{1}{\sqrt m}F_m^\ast(Pf).
\end{equation}
Let $\overline\partial: \Omega^{0, q}(\tilde D)\rightarrow\Omega^{0, q+1}(\tilde D)$ be the Cauchy-Riemmann operator and we have
$$\overline\partial=\sum_{j=1}^{n-1}e_j(z)\wedge\overline L_j+\sum_{j=1}^{n-1}(\overline\partial e_j)
\left(z\right)\wedge\left(e_j\left(z\right)\wedge\right)^\ast$$
where $\left(e_j\left(z\right)\wedge\right)^\ast:T^{\ast0,q}\tilde D\To
T^{\ast0,q-1}\tilde D$ is the adjoint of $e_j\left(z\right)\wedge$ with respect to the Hermitian metric $\langle\cdot\big|\cdot\rangle$ on $T^{\ast 0, q}\tilde D$ given in Remark~\ref{g1},
$ j=1,\ldots, n-1$. That is, $$\left\langle e_j\left(z\right)\wedge u \Big|v \right\rangle=\left\langle u\Big|\left(e_j\left(z\right)
\wedge\right)^\ast v\,\right\rangle$$ for all $u\in T^{\ast0,q-1}\tilde D$, $v\in T^{\ast0,q}\tilde D$.

The scaled differential operator $\overline\partial_{(m)}: F_m^\ast\Omega^{0,q}(\tilde D_{\log m})\rightarrow F_m^\ast\Omega^{0,q+1}(\tilde D_{\log m})$ is given by
\begin{equation}\label{ee1}
\overline\partial_{(m)}=\sum_{j=1}^{n-1}e_j\left(\frac{z}{\sqrt m}\right)\wedge\overline L_{j,(m)}+\sum_{j=1}^{n-1}\frac{1}{\sqrt{m}}(\overline\partial e_j)
\left(\frac{z}{\sqrt m}\right)\wedge\left(e_j\left(\frac{z}{\sqrt m}\right)\wedge\right)^\ast.
\end{equation}
Similarly, $\left(e_j\left(\frac{z}{\sqrt
m}\right)\wedge\right)^\ast:F^\ast_mT^{\ast0,q}X\To
F^\ast_mT^{\ast0,q-1}\tilde D$ is the adjoint of $e_j\left(\frac{z}{\sqrt m}\right)\wedge$ with respect to
$\langle\cdot\big|\cdot\rangle_{F_m^\ast}$, $ j=1,\ldots, n-1$.
From (\ref{j9}) and (\ref{ee1}), $\overline\partial_{(m)}$ satisfies that
\begin{equation}\label{l1}
\overline\partial_{(m)}F_m^\ast f=\frac{1}{\sqrt{m}}F_m^\ast(\overline\partial f),\ \ \forall f\in\Omega^{0,q}(F_m(\tilde D_{\log m})).
\end{equation}
Let $(\cdot|\cdot)_{2mF_m^\ast\varphi}$ be the weighted inner product on the space $F_m^\ast\Omega^{0,q}_0(\tilde D_{\log m})$ defined  as follows:
\begin{equation}
(f|g)_{2mF_m^\ast\varphi}=\int_{\tilde D_{\log m}}\langle f|g\rangle_{F_m^\ast}e^{-2mF_m^\ast\varphi}\lambda(\frac{z}{\sqrt m})dv(z).
\end{equation}
Let $\overline\partial^\ast_{ (m)}: F_m^\ast\Omega^{0,q+1}(\tilde D_{\log m})\rightarrow F_m^\ast\Omega^{0,q}(\tilde D_{\log m})$
be the formal adjoint of $\overline\partial_{(m)}$ with respect to $(\cdot|\cdot)_{2mF_m^\ast\varphi}$.
Let $\overline\partial^{\ast, 2m\varphi}: \Omega^{0, q+1}(\tilde D)\rightarrow\Omega^{0, q}(\tilde D)$
be the formal adjoint of $\overline\partial$ with respect to the weighted inner product $(\cdot|\cdot)_{2m\varphi}$. Then we also have
\begin{equation}\label{l2}
\overline\partial_{(m)}^\ast F_m^\ast f=\frac{1}{\sqrt{m}}F_m^\ast(\overline\partial^{\ast, 2m\varphi} f),\ \ \forall f\in\Omega^{0,q}(F_m(\tilde D_{\log m})).
\end{equation}
We now define the scaled complex Laplacian $\Box^{(q)}_{(m)}:F_m^\ast\Omega^{0,q}(\tilde D_{\log m})\rightarrow F_m^\ast\Omega^{0,q}(\tilde D_{\log m})$ which is given by
$
\Box^{(q)}_{(m)}=\overline\partial_{(m)}^\ast\overline\partial_{(m)}
+\overline\partial_{(m)}\overline\partial_{(m)}^\ast.
$
Then (\ref{l1}) and (\ref{l2}) imply that
\begin{equation}\label{s}
\Box^{(q)}_{(m)}F_m^\ast f=\frac{1}{m}F_m^\ast(\Box^{(q)}_{2m\varphi}f), \forall f\in\Omega^{0, q}(F_m(\tilde D_{\log m})).
\end{equation}
Here,
\begin{equation}\label{e-gue150620f}
\Box^{(q)}_{2m\varphi}=\overline\partial\,\overline\partial^{\ast, 2m\varphi}+
\overline\partial^{\ast, 2m\varphi}\overline\partial:\Omega^{0, q}(\tilde D)\rightarrow\Omega^{0, q}(\tilde D)
\end{equation}
is the complex Laplacian with respect to the given Hermitian metric on $T^{\ast 0, q}\tilde D$ and weight function $2m\varphi(z)$ on $\tilde D$. Since $2mF_m^\ast\varphi=2\Phi_0(z)+\frac{1}{\sqrt m}O(|z|^3), \forall z\in\tilde D_{\log m}$, where $\Phi_0(z)=\sum_{j=1}^{n-1}\lambda_j|z_j|^2$, we have
\begin{equation}\label{j6}
\lim_{m\rightarrow\infty}\sup_{\tilde D_{\log m}}|\partial_x^{\alpha}(2mF_m^\ast\varphi-2\Phi_0)|=0, \forall \alpha\in \mathbb N_0^{2n-2}.
\end{equation}

Consider $\mathbb C^{n-1}$. Let $\langle\cdot|\cdot\rangle_{\mathbb C^{n-1}}$ be the Hermitian metric on $T^{\ast0,q}\mathbb C^{n-1}$
such that
\[\{d\overline z^J: |J|=q, \mbox{$J$ is strictly increasing}\}\]
is an orthonormal basis. Let $(\cdot|\cdot)_{2\Phi_0}$ be the $L^2$ inner product on $\Omega^{0,q}_0(\mathbb C^{n-1})$ given by
\begin{equation}\label{e-gue150617}
(f|g)_{2\Phi_0}=\int_{\mathbb C^{n-1}}\langle f|g \rangle e^{-2\Phi_0(z)}dv(z),\ \ f, g\in\Omega^{0,q}_0(\mathbb C^{n-1}).
\end{equation}
Put
\begin{equation}\label{e-gue150617I}
\Box^{(q)}_{2\Phi_0}=\ddbar\,\overline{\partial}^{\ast,2\Phi_0}+\overline{\partial}^{\ast,2\Phi_0}\ddbar: \Omega^{0,q}(\mathbb C^{n-1})\To\Omega^{0,q}(\mathbb C^{n-1}),
\end{equation}
where $\overline{\partial}^{\ast,2\Phi_0}$ is the formal adjoint of $\ddbar$ with respect to $(\cdot|\cdot)_{2\Phi_{0}}$.

From \eqref{j6}, it is not difficult to check that
\begin{equation}\label{j7}
\Box^{(q)}_{(m)}=\Box^{(q)}_{2\Phi_0}+\epsilon_m \mathcal P_m\ \ \mbox{on $\tilde D_{\log m}$},
\end{equation}
where $\mathcal P_m$ is a second order partial differential operator and all the coefficients of $\mathcal P_m$ are uniformly bounded with respect to $m$ in $C^\mu(\tilde D_{\log m})$ norm for every $\mu\in\mathbb N_0$ and $\epsilon_m$ is a sequence tending to zero as $m\rightarrow\infty.$ By the convergence property of (\ref{j6}) and (\ref{j7}), we have Garding's inequality for elliptic operator $\Box^{(q)}_{(m)}$.

\begin{proposition}\label{k2}
For every $r>0$ with $\tilde D_{2r}\subset \tilde D_{\log m}$ and $s\in\mathbb N_0$, there is a constant $C_{r, s}>0$ independent of $m$ and the point $x_0$ such that
\begin{equation}
\|u\|^2_{2mF_m^\ast\varphi, s+2, \tilde D_{r}}\leq C_{s, r}\left(\|u\|^2_{2mF_m^\ast\varphi, \tilde D_{2r}}+\|\Box^{(q)}_{(m)}u\|^2_{2mF_m^\ast\varphi, s, \tilde D_{2r}}\right)
\end{equation}
for all $ u\in F_m^\ast\Omega^{0, q}(\tilde D_{\log m}),$ where $\|u\|_{2mF_m^\ast\varphi, s, \tilde D_{r}}$ is the weighted Sobolev norm of order $s$ with respect to the weight function $2mF_m^\ast\varphi$ which is given by
\begin{equation}
\|u\|^2_{2mF_m^\ast\varphi, s, \tilde D_r}=\sum^\prime\nolimits_{\alpha\in\mathbb N_0^{2n-2}, |\alpha|\leq s, |J|=q}\int_{\tilde D_r}|\partial^{\alpha}_xu_J|^2e^{-2mF_m^\ast\varphi}\lambda(\frac{z}{\sqrt{m}})dv(z),
\end{equation}
where $u=\sum_{|J|=q}^\prime u_Je^J(\frac{z}{\sqrt m})\in F_m^\ast\Omega^{0, q}(\tilde D_{\log m}).$
\end{proposition}

\section{Morse inequalities on CR manifolds}\label{a6}
Let $f_1, \ldots, f_{d_m}$ be an orthonormal basis of $\mathcal H^q_{b, m}(X)$. The Szeg\"o kernel of $\mathcal H^q_{b, m}(X)$ is defined by
\begin{equation}
\Pi^{q}_{ m}(x):=\sum_{j=1}^{d_m}|f_j(x)|^2.
\end{equation}
It is easy to see that $\Pi^{q}_{m}(x)$ is independent of the choice of the orthonormal basis and
\begin{equation}
\dim\mathcal H^q_{b,m}(X)=\int_X \Pi^{q}_{m}(x)dv_X.
\end{equation}
We denote by $X(q)$ a subset of $X$ such that $$X(q):=\{x\in X: \mathcal L_x~\text{has exactly}~q~\text{negative eigenvalues and }~n-1-q~ \text{positive eigenvalues}\}.$$
Recall that for every $k\in\mathbb N$, $X_k$ is given by \eqref{e-gue150614}.
The following is our first technical result.

\begin{theorem}\label{m}
Let $X$ be a compact connected CR manifold with a transversal CR $S^1$-action.
For every $q=0,1,2,\ldots,n-1$, we have
\begin{equation}\label{e-gue150619dI}
\sup\{m^{-(n-1)}\Pi_{m}^{q}(x):  m\in\mathbb N, x\in X\}<\infty
\end{equation}
and for every $k\in\mathbb N$ with $X_k\neq \emptyset$, we have
\begin{equation}
\limsup_{m \rightarrow\infty}m^{-(n-1)}\Pi^{q}_{m}(x)\leq \frac{k}{2\pi^{n}}|\det\mathcal L_x|\cdot1_{X(q)}(x), \forall x\in X_k,
\end{equation}
where $1_{X(q)}(x)$ denotes the characteristic function of the subset $X(q)\subset X$ and $\mathcal L_x$ is the Levi-form defined in Definition~\ref{d-1.2}.

In particular, for every $q=0,1,2,\ldots,n-1$, we have
\begin{equation}\label{k6}
\limsup_{m \rightarrow\infty}m^{-(n-1)}\Pi^{q}_{m}(x)\leq \frac{1}{2\pi^{n}}|\det\mathcal L_x|\cdot1_{X(q)}(x), \forall x\in X_{\rm reg}.
\end{equation}

\end{theorem}
\subsection{Main results}
From Lemma~\ref{g2} and Theorem~\ref{m} and by Fatou's lemma we obtain the weak Morse inequalities

\begin{theorem}[weak Morse inequalities]\label{a4}
Let $X$ be a compact connected CR manifold with a transversal CR $S^1$-action. Assume that ${\rm dim}_{\mathbb R}X=2n-1, n\geq 2.$ Then for every $q=0,1,2,\ldots,n-1$, we have
\begin{equation}\label{a1}
\dim H^q_{b, m}(X)\leq\frac{m^{n-1}}{2\pi^n}\int_{X(q)}|\det\mathcal L_x|dv_X(x)+o(m^{n-1}), m\rightarrow\infty.
\end{equation}
\end{theorem}

From Theorem~\ref{a4}  we deduce Demailly's weak holomorphic Morse inequalities (see \cite[Theorem 0.1]{D85}  and \cite[Theorem 1.7.1]{MM07}):

\begin{corollary}[Demailly's weak morse inequalities]\label{kkkk}
Let $M$ be a compact Hermitian manifold of dimension ${\rm dim}_{\mathbb C}M=n$ and let $(L, h^L)$ be a Hermitian line bundle over $M$.  Then for $q=0,1,2,\ldots,n$
\begin{equation}
\dim H_{\overline\partial}^{q}(M, L^k)\leq \frac{k^{n}}{(2\pi)^{n}}\int_{M(q)}|\det \mathcal R^L_x|dv_M(x)+o(k^{n}), k\rightarrow\infty,
\end{equation}
where $H_{\overline\partial}^{q}(M, L^k)$ denotes the $q$-th $\ddbar$-cohomology group with values in $L^k$, $dv_M$ is the induced volume form on $M$, $\mathcal R^L_x, x\in M$ is the Chern curvature of the Hermitian line bundle $(L, h^L)$ and  $M(q)$ is a subset of $M$ where $\mathcal R^L_x$ has exactly $q$ negative eigenvalues and $n-q$ positive eigenvalues.
\end{corollary}

\begin{proof}
We take $X$ to be the circle bundle $\{v\in L^\ast : |v|^2_{h^{-1}}=1\}$ over the compact complex manifold $M$  where $(L^{\ast}, h^{-1})$ is the dual line bundle of Hermitian line bundle $(L, h)$ over $M$. Let $(z, \lambda)$ be the local coordinates on $L^\ast$, where $\lambda$ is the fiber coordinates. The natural $S^1$- action on $X$ is defined by $e^{i\theta}\circ(z, \lambda)=(z, e^{i\theta}\lambda)$. Then we can check that $X$ is a compact CR manifold with a transversal CR $S^1$-action. On $X$ we can check that $\mathcal L_x|_{T^{1, 0}X}=\frac12\mathcal R^L_z$ where $x=(z, \lambda).$ It is well known that (see \cite[p.746]{MM06}, Theorem 1.2 in~\cite{CHT15}) for $k\in\mathbb N$, ${\rm dim}H^q_{\overline\partial}(M, L^k)={\rm dim}H^q_{b, k}(X)$. From (\ref{a1}) we have
\begin{equation}
\begin{split}
&{\rm dim}H^q_{\overline\partial}(M, L^k)=\dim H^q_{b, k}(X)\\
&\leq\frac{k^{n}}{2\pi^{n+1}}\int_{X(q)}|\det\mathcal L_x|dv_X(x)+o(k^{n})\\
&\leq \frac{k^{n}}{2\pi^{n+1}}\times\frac{2\pi}{2^{n}}\int_{M(q)}|\det \mathcal R^L|dv_M+o(k^{n})\\
&\leq \frac{k^{n}}{(2\pi)^{n}}\int_{M(q)}|\det \mathcal R^L_x|dv_M(x)+o(k^{n}).
\end{split}
\end{equation}
Thus, we get the conclusion of Corollary~\ref{kkkk}.
\end{proof}

It should be noticed that the relation between sections of the holomorphic line bundle and function theory on the associate
Grauert tube was first observed by Grauert~\cite{G62}. The isomorphism of the subcomplex $(\Omega^{0,\bullet}_m(X),\ddbar_b)$
to the Dolbeault complex $(\Omega^{0,\bullet}(M,L^m),\ddbar)$ was established by Ma-Marinescu~\cite[p.746]{MM06}.

For $\sigma>0$, we collect the eigenspace of $\Box^{(q)}_{b,m}$ whose eigenvalue is less than or equal to $\sigma$ (see Theorem~\ref{gI}) and define
\begin{equation}\label{c4}
\mathcal H^q_{b, m, \leq \sigma}(X):=
\bigoplus\limits_{\lambda\leq \sigma}\mathcal H^q_{b, m,\lambda}(X).
\end{equation}

The Szeg\"o kernel function of the space $\mathcal H^q_{b, m, \leq m\sigma}(X)$ is defined by $\Pi^{q}_{m, \leq m\sigma}(x)=\sum_{j=1}^{d_m}|g_j(x)|^2$, where $\{g_j(x)\}_{j=1}^{d_m}$ is any orthonormal basis for the space $\mathcal H^q_{b, m, \leq m\sigma}(X).$ Our second main technique result is the following

\begin{theorem}\label{o}
For any sequence $v_m>0$ with $v_m\rightarrow0$ as $m\rightarrow\infty$, there exists a constant $C>0$ independent of $m$ and $x\in X$ such that
\begin{equation}\label{e-gue150619d}
m^{-(n-1)}\Pi^{q}_{m, \leq mv_m}(x)\leq C,\ \ \forall m\in\mathbb N,\ \ \forall x\in X.
\end{equation}

Moreover, there is a sequence $\delta_m>0, \delta_m\rightarrow0$ as $m\rightarrow\infty$, such that for any sequence $v_m>0$ with $\lim\limits_{m\rightarrow\infty}\frac{\delta_m}{v_m}=0$, we have
\begin{equation}\label{p}
\lim\limits_{m\rightarrow\infty}m^{-(n-1)}\Pi^{q}_{m, \leq mv_m}(x)=\frac{1}{2\pi^n}|\det\mathcal L_x|\cdot1_{X(q)}, \forall x\in X_{\rm reg}.
\end{equation}

Since the measure of $X\setminus X_{{\rm reg}}=0$, integrating (\ref{p}) and by Fatou's Lemma we have
\begin{equation}
{\rm dim}\mathcal H^q_{m, \leq mv_m}(X)=\frac{m^{n-1}}{2\pi^n}\int_{X(q)}|\det\mathcal L_x|dv_X(x)+o(m^{n-1}), m\rightarrow\infty.
\end{equation}
\end{theorem}
From Theorem \ref{o} and the linear algebraic argument from Demailly in \cite{D85}, \cite{D91} and \cite{M96}, we obtain the strong Morse inequalities

\begin{theorem}[strong Morse inequalities]\label{a2}
Let $X$ be a compact connected CR manifold with a transversal CR $S^1$-action. Assume that ${\rm dim}_{\mathbb R}X=2n-1, n\geq 2.$ For every $q=0,1,2,\ldots,n-1$, as $m\rightarrow\infty$, we have
\begin{equation}
\sum_{j=0}^q(-1)^{q-j}{\rm dim}H^j_{b, m}(X)\leq\frac{m^{n-1}}{2\pi^n}\sum_{j=0}^q(-1)^{q-j}\int_{X(j)}|\det\mathcal L_x|dv_X(x)+o(m^{n-1}).
\end{equation}
In particular, when $q=n-1$, as $m\rightarrow\infty$, we have the asymptotic Riemann-Roch theorem
\begin{equation}
\sum_{j=0}^{n-1}(-1)^{j}{\rm dim}H^j_{b, m}(X)=\frac{m^{n-1}}{2\pi^n}\sum_{j=0}^{n-1}(-1)^{j}\int_{X(j)}|\det\mathcal L_x|dv_X(x)+o(m^{n-1}).
\end{equation}
\end{theorem}

From Theorem~\ref{a2}, we can repeat the procedure in the proof of Corollary~\ref{kkkk} with minor change and get Demailly's strong Morse inequalities. We refer the readers to the book by Ma-Marinescu~\cite{MM07} for the heat kernel approach of Demailly's Morse inequalities.

When $q=1$, from the strong Morse inequality in Theorem \ref{a2} we have
\begin{equation}\label{a3}
\begin{split}
&-{\rm dim}H^0_{b, m}(X)+{\rm dim}H^1_{b, m}(X)\\
&\leq\frac{m^{n-1}}{2\pi^n}\int_{X(1)}|\det\mathcal L_x|dv_X(x)-\frac{m^{n-1}}{2\pi^n}\int_{X(0)}|\det\mathcal L_x|dv_X(x)+o(m^{n-1}), m\rightarrow\infty.
\end{split}
\end{equation}
The inequality (\ref{a3}) implies that

\begin{theorem}[Grauert-Riemenschneider criterion]\label{a5}
Let $X$ be a compact connected CR manifold with a transversal CR $S^1$-action. Assume that ${\rm dim}_{\mathbb R}X=2n-1, n\geq 2.$ If $X$ is weakly pseudoconvex and strongly pseudoconvex at a point, then as $m\rightarrow\infty$
\begin{equation}\label{x4}
{\rm dim}H^0_{b, m}(X)\thickapprox m^{n-1},~ {\rm dim}H^q_{b, m}(X)=o(m^{n-1})~\text{for}~q\geq 1.
\end{equation}
In particular, we have ${\rm dim}H^0_{b}(X)=\infty$.
\end{theorem}

\begin{proof}
Since $X$ is a weakly pseudoconvex manifold and strongly pseudoconvex at least at a point, we have
$X(q)=\emptyset$ for every $q\geq 1$ and $X(0)$ contains a ball. By the weak Morse in Theorem~\ref{a4} we have ${\rm dim}H^q_{b, m}(X)=o(m^{n-1})~\text{for}~q\geq 1$ as $m\rightarrow\infty$. By the weak Morse inequalities for $q=0$ and (\ref{a3}), we get the conclusion of (\ref{x4}). Using the Fourier decomposition (\ref{x6}), we have ${\rm dim}H^0_{b}(X)=\infty.$
\end{proof}
Theorem~\ref{a5} implies a new proof of Grauert-Riemenschneider conjecture (Siu's
criterion) as stated and solved by Siu \cite{S84}, ~\cite[Theorem\,1]{S85} and Ma-Marinescu~\cite[Theorem
2.2.27 (i)]{MM07}.
\begin{corollary}[Grauert-Riemenschneider conjecture, Siu's criterion]
Let $M$ be a compact Hermitian manifold and let $(L, h^L)$ be a Hermitian line bundle over $M$. If $L$ is semi-positive and positive at a point, then $L$ is big.
\end{corollary}
\begin{proof}
Applying Theorem ~\ref{a5} to the circle bundle of $L$ we get the conclusion of the corollary.
\end{proof}
Demailly \cite[Theorem 0.8(a)]{D85} and \cite[Theorem 2.2.27(ii)]{MM07} proved a more general form of the Grauert-Riemenschneider conjecture, namely that if the
integral of $c_1(L, h)^n$ over the set of points for which $c_1(L, h)$ has one or fewer
negative eigenvalues is positive (i.e. $\int_{M(\leq 1)} c_1(L, h)^n > 0$), then $L$ is big and $M$ is Moishezon.

If we set $\int_{X(\leq 1)}|{\rm det}\mathcal L_x|dv_X=\int_{X(0)}|{\rm det}\mathcal L_x|dv_X-\int_{X(1)}|{\rm det}\mathcal L_x|dv_X$
and assume that the Levi form of CR manifold is not always semi-positive but that the integral $\int_{X(\leq 1)}|{\rm det}\mathcal L_x|dv_X>0$, then by (\ref{a3}) we still get many CR functions:

\begin{theorem}\label{x3}
Let X be a compact connected CR manifold of dimension $2n-1$ with a
transversal CR  $S^1$-action. Assume that
\begin{equation}
\int_{X(\leq 1)}|{\rm det}\mathcal L_x|dv_X>0.
\end{equation}
Then ${\rm dim}H^0_{b, m}(X)\approx m^{n-1}$ as $m\rightarrow\infty.$ In particular, ${\rm dim}H^0_{b}(X)=\infty.$
\end{theorem}

Theorem~\ref{x3} is the analogue of Demailly's criterion \cite[Theorem 0.8(a)]{D85} and \cite[Theorem 2.2.27(ii)]{MM07}, for CR manifolds with transversal CR $S^1$-action and actually implies this criterion if applied to the Grauert tube. Theorem~\ref{x3} shows that one can allow the Levi form to be negative in a controlled way and still have a lot of CR functions.

\subsubsection{Morse inequalities for $m\To-\infty$}

In the main results above, we only consider the Morse inequalities for the positive Fourier component ${\rm dim}H^q_{b, m}(X)$ as $m\rightarrow\infty.$ In fact, we also have the Morse inequalities for the negative Fourier component $H^q_{b, m}(X)$ as $m\rightarrow-\infty$. Based on the same arguments as in the proof of the main results, the bounds for $H^q_{b, m}(X)$ for $m\rightarrow-\infty$ will be given in terms of integrals of the Levi form of $X$ over the sets $X(n-1-q)$. More precisely, we have

\begin{theorem}\label{t-gue160214}
Let $X$ be a compact connected CR manifold with a transversal CR $S^1$-action. Assume that ${\rm dim}_{\mathbb R}X=2n-1, n\geq 2.$ For every $q=0,1,2,\ldots,n-1$, as $m\rightarrow-\infty$, we have
\begin{equation}
\begin{split}
&\dim H^q_{b, m}(X)\leq\frac{\abs{m}^{n-1}}{2\pi^n}\int_{X(n-1-q)}|\det\mathcal L_x|dv_X(x)+o(\abs{m}^{n-1}),\\
&\sum_{j=0}^q(-1)^{q-j}{\rm dim}H^j_{b, m}(X)\leq\frac{\abs{m}^{n-1}}{2\pi^n}\sum_{j=0}^q(-1)^{q-j}\int_{X(n-1-j)}|\det\mathcal L_x|dv_X(x)+o(\abs{m}^{n-1}).
\end{split}
\end{equation}
In particular, when $q=n-1$, as $m\rightarrow-\infty$, we have the asymptotic Riemann-Roch theorem
\begin{equation}
\sum_{j=0}^{n-1}(-1)^{j}{\rm dim}H^j_{b, m}(X)=\frac{\abs{m}^{n-1}}{2\pi^n}\sum_{j=0}^{n-1}(-1)^{j}\int_{X(n-1-j)}|\det\mathcal L_x|dv_X(x)+o(\abs{m}^{n-1}).
\end{equation}
\end{theorem}

From Theorem~\ref{a4}, Theorem~\ref{a2} and Theorem~\ref{t-gue160214}, we deduce

\begin{theorem}
Let $X$ be a compact connected CR manifold of real dimension $2n-1$ with a transversal CR $S^1$-action. Let $q\in\{0, 1, \cdots, n-1\}.$ Assume that the Levi form of $X$ has $q$ non-positive and $n-1-q$ non-negative eigenvalues everywhere. Then
\begin{equation}
\begin{split}
&{\rm dim}H^j_{b, m}(X)=o(m^{n-1}), ~\text{as}~m\rightarrow\infty, ~\text{for}~j\neq q\\
&{\rm dim}H^j_{b, m}(X)=o(|m|^{n-1}),~\text{as}~m\rightarrow -\infty, ~\text{for}~j\neq n-1-q.
\end{split}
\end{equation}
If moreover the Levi-form is non-degenerate at some point, then
\begin{equation}
\begin{split}
&{\rm dim}H^q_{b, m}(X)\approx m^{n-1}, ~\text{as}~m\rightarrow\infty\\
&{\rm dim}H^{n-1-q}_{b, m}(X)\approx |m|^{n-1}, ~\text{as}~m\rightarrow -\infty\\
&{\rm dim}H^q_{b}(X)=\infty, {\rm dim}H^{n-1-q}_{b}(X)=\infty.
\end{split}
\end{equation}

In particularly, if $X$ is weakly pseudoconvex and strongly pseudoconvex at a point, then
\[\mbox{${\rm dim}H^{n-1}_{b, m}(X)\approx |m|^{n-1}$ as $m\rightarrow -\infty$}\]
 and in particluar ${\rm dim}H^{n-1}_{b}(X)=\infty$. Moreover,
${\rm dim}H^q_{b}(X)=o(\abs{m}^{n-1})$ as $m\To-\infty$ for $q\leq n-2$.
\end{theorem}

\subsection{Proofs of Theorem~\ref{m} and the weak Morse inequalities}

Fix $x_0\in X$ and choose canonical local patch $D$ near $x_0$ with canonical coordinates $(z, \theta,\varphi)$ such that $(z,\theta,\varphi)$ is trivial at $x_0$. Write $D=\tilde D\times(-\delta, \delta), \tilde D=\{z\in\mathbb C^{n-1}: |z|<\varepsilon\}$. In this section, we always treat $\tilde D$ as an open subset of $\mathbb C^{n-1}$ with the complex coordinates $z=(z_1, \ldots, z_{n-1})$. We choose the fixed Hermitian metric given on $T^{\ast 0, 1}\tilde D$ defined in Remark~\ref{g1} and extend it to $T^{\ast 0, q}\tilde D.$ We still use the notation $\langle\cdot|\cdot\rangle$ to denote the Hermitian metric on $T^{\ast 0, q}\tilde D.$
Let $u\in\Omega^{0, q}_m(X)$. From the definition of $\Omega^{0, q}_m(X)$ we have that $Tu=imu$. Then on $D$, $u=\tilde u(z)e^{im\theta}$ with $\tilde u(z)\in\Omega^{0, q}( \tilde D)$ and $\tilde u(z)=\sum_{|J|=q}^\prime\tilde u_J(z)d\overline z^J$. Before the proof of the weak Morse inequalities, we first need the following lemma

\begin{lemma}\label{b1}
For all $u\in\Omega_m^{0, q}(X)$, on $D$ we have
\begin{equation}
\begin{split}
&\overline\partial_bu=e^{im\theta}e^{-m\varphi}\overline\partial(e^{m\varphi}e^{-im\theta}u),  \overline\partial_b^\ast u=e^{im\theta}e^{-m\varphi}\overline\partial^{\ast, 2m\varphi }(e^{m\varphi}e^{-im\theta}u),\\
&\Box^{(q)}_{b, m}u=e^{im\theta}e^{-m\varphi}\Box^{(q)}_{2m\varphi}(e^{m\varphi}e^{-im\theta}u).
\end{split}
\end{equation}
Recall that $\overline\partial^{\ast, 2m\varphi}$ is as in the discussion before \eqref{l2} and $\Box^{(q)}_{2m\varphi}$ is given by \eqref{e-gue150620f}.
\end{lemma}

\begin{proof}
Let $u=\sum_{|J|=q}^\prime u_Jd\overline z^J$. Then $\overline\partial_b u=\sum_{|J|=q}^\prime\left(\frac{\partial u_J}{\partial\overline z_j}-i\frac{\partial\varphi(z)}{\partial\overline z_j}\frac{\partial u_J}{\partial\theta}\right)d\overline z_j\wedge d\overline z^J.$ By the assumption of the Lemma~\ref{b1}, $Tu=imu$ which implies that on $\frac{\partial u_J}{\partial\theta}=imu_J$ on $D$ for every $J$. Then
\begin{equation}\label{b3}
\begin{split}
\overline\partial_b u&=\sum\nolimits_{|J|=q}^\prime\sum_{j=1}^{n-1}\left(\frac{\partial u_J}{\partial\overline z_j}+m\frac{\partial\varphi(z)}{\partial\overline z_j}u_J\right)d\overline z_j\wedge d\overline z^J\\
&=e^{im\theta}\sum\nolimits_{|J|=q}^\prime\sum_{j=1}^{n-1}\left(\frac{\partial \tilde u_J}{\partial\overline z_j}+m\frac{\partial\varphi(z)}{\partial\overline z_j}\tilde u_J\right)d\overline z_j\wedge d\overline z^J,
\end{split}
\end{equation}
where $u=e^{im\theta}\tilde u(z)$ on $D$, $\tilde u(z)=\sideset{}{'}\sum_{|J|=q}\tilde u_J(z)d\overline z^J$.
Set $v(z)=e^{m\varphi}\tilde u(z)=\sideset{}{'}\sum_{|J|=q}v_J(z)d\overline z^J$. Then
\begin{equation}\label{b2}
\frac{\partial v_J}{\partial\overline z_j}=\frac{\partial}{\partial\overline z_j}(\tilde u_J(z)e^{m\varphi})=e^{m\varphi}\left(\frac{\partial\tilde u_J}{\partial\overline z_j}+m\frac{\partial\varphi(z)}{\partial\overline z_j}\tilde u_J\right).
\end{equation}
Substituting (\ref{b2}) to (\ref{b3}) we get the conclusion of the first identity of Lemma~\ref{b1}.

Since $\overline{\partial}^\ast_bu\in\Omega^{0,q-1}_m(X)$, on $D$, we write $\overline{\partial}^\ast_bu=e^{im\theta}\tilde v(z)$, $\tilde v(z)\in\Omega^{0,q-1}(\tilde D)$. Take $\chi(\theta)\in C^\infty_0((-\delta,\delta))$ with $\int\chi(\theta)d\theta=1$. Let $g\in\Omega^{0,q-1}_0(\tilde D)$. We have
\begin{equation}\label{e-gue150620fa}
\begin{split}
&(\overline{\partial}^\ast_bu|e^{-2m\varphi(z)}g(z)\chi(\theta)e^{im\theta})=(e^{im\theta}\tilde v(z)|e^{-2m\varphi(z)}g(z)\chi(\theta)e^{im\theta})=(\tilde v(z)|g(z))_{2m\varphi}.
\end{split}
\end{equation}
On the other hand, from the proof first identity of Lemma~\ref{b1}, we have
\begin{equation}\label{e-gue150620faI}
\begin{split}
&(\overline{\partial}^\ast_bu|e^{-2m\varphi(z)}g(z)\chi(\theta)e^{im\theta})
=(u|\ddbar_b(e^{-2m\varphi(z)}g(z)\chi(\theta)e^{im\theta}))\\
&=(u|\chi(\theta)e^{im\theta}e^{-m\varphi(z)}\ddbar(e^{-m\varphi(z)}g(z)))
+(u|(-i)\chi^{\prime}(\theta)e^{im\theta}e^{-2m\varphi(z)}
\overline\partial\varphi\wedge g(z))\\
&=(e^{m\varphi(z)}\tilde u(z)|\ddbar(e^{-m\varphi(z)}g(z))_{2m\varphi}=(\overline{\partial}^{\ast, 2m\varphi}(e^{m\varphi}\tilde u)|e^{-m\varphi(z)}g(z))_{2m\varphi}\\
&=(e^{-m\varphi(z)}\overline{\partial}^{\ast, 2m\varphi}(e^{m\varphi}\tilde u)|g(z))_{2m\varphi}.
\end{split}
\end{equation}
From \eqref{e-gue150620fa} and \eqref{e-gue150620faI}, we get
\[(\tilde v(z)|g(z))_{2m\varphi}=(e^{-m\varphi(z)}\overline{\partial}^{\ast, 2m\varphi }(e^{m\varphi}\tilde u)|g(z))_{2m\varphi},\ \ \forall g\in\Omega^{0,q}_0(\tilde D),\]
and hence
\[e^{-im\theta}\ddbar^\ast_bu=\tilde v=e^{-m\varphi(z)}\overline{\partial}^{\ast, 2m\varphi}(e^{m\varphi}\tilde u)\ \ \mbox{on $D$}.\]
We get the second identity in Lemma~\ref{b1}. The third identity can be deduced directly from the other two identities.

\end{proof}

For any $u\in\Omega^{0, q}(X), u=\sum_{|J|=q}^\prime u_J(z, \theta)e^J(z).$ Here $J=(j_1, \cdots, j_q)$ with $1\leq j_1<\cdots<j_q\leq n-1$, $e^J=e^{j_1}\wedge\cdots\wedge e^{j_q}$ and $\{e_j\}_{j=1}^{n-1}$ is the orthonormal frame chosen in Lemma~\ref{b5}. Set $S^q_{m, J}(x)=\sup\limits_{u\in\mathcal H^q_{b, m}(X), \|u\|=1}|u_J(x)|^2$ which is the extremal function along the direction $e^J$. We can repeat the proof of Lemma 2.1 in~\cite{HM12} and conclude that
\begin{lemma}\label{l-b6}
For every local orthonormal frame $\{e^J:|J|=q,~\text{strictly increasing}\}$ of $T^{\ast0,q}X$ over an open set $D$, we have for $y\in D$
\begin{equation}
\Pi_{m}^q(y)=\sum\nolimits_{|J|=q}^\prime S^q_{m, J}(y).
\end{equation}
\end{lemma}

Now we are going to prove Theorem~\ref{m} and the weak Morse inequality.

\begin{proof}
Fix $x_0\in X$ and choose canonical local patch $D=\{(z, \theta): |z|<\varepsilon, |\theta|<\delta\}$ with canonical coordinates $(z, \theta,\varphi)$ such that $(z,\theta,\varphi)$ is trivial at $x_0$. For any $u\in\mathcal{H}^q_{b, m}(X)$ with $\|u\|=1$, on $D$ we have
$u(z, \theta)=\tilde u(z)e^{im\theta}$. Set $v_m(z)=e^{m\varphi(z)}\tilde u(z), z\in D.$ Then from Lemma~\ref{b1}, $\Box^{(q)}_{b, m}u(z, \theta)=0$ and $\|u\|=1$ we deduce that on $D$,
\begin{equation}\label{j5}
\Box^{(q)}_{2m\varphi}v_m(z)=0~\text{and}~\int_{\tilde D}|v_m(z)|^2e^{-2m\varphi(z)}\lambda(z)dv(z)\leq\frac{1}{2\delta},
\end{equation}
where $\Box^{(q)}_{2m\varphi}$ is as in \eqref{e-gue150620f}.
Set $\tilde v_{(m)}(z)=m^{\frac{-(n-1)}{2}}e^{m\varphi(\frac{z}{\sqrt m})}\tilde u(\frac{z}{\sqrt m})$. Then by (\ref{s}) and (\ref{j5}) we have
\begin{equation}\label{k1}
\Box^{(q)}_{(m)}\tilde v_{(m)}(z)=0 ~\text{and}~\int_{\tilde D_{r }}|\tilde v_{(m)}(z)|^2e^{-2m\varphi(\frac{z}{\sqrt m})}\lambda(\frac{z}{\sqrt{m}})dv(z)\leq\frac{1}{2\delta},
\end{equation}
for any $r<\log m$ when $m$ is large. From Proposition~\ref{k2} and (\ref{k1}), there exists a constant $C_{r, s, \delta}^\prime>0$ independent of $m$ and the point $x_0$ such that
\begin{equation}\label{k3}
\|\tilde v_{(m)}\|^2_{2mF_m^\ast\varphi, s+2, \tilde D_r}\leq C^{\prime}_{r, s, \delta}.
\end{equation}
Since $X$ is compact we can choose $\delta$ which is independent of $x_0$. For $s\geq 2n-2$, from (\ref{k3}) and by Sobolev embedding theorem, there exists a constant $C^\prime$ which is independent of $x_0, m$ such that
\begin{equation}\label{k4}
m^{-(n-1)}|u(x_0)|^2=m^{-(n-1)}|\tilde u(0)|^2=|\tilde v_{(m)}(0)|^2\leq C^\prime.
\end{equation}
From (\ref{k4}) and Lemma~\ref{l-b6}, we get the conclusion of the first part of Theorem~\ref{m}.

Fix $|J|=q$, $J$ is strictly increasing. There exists a sequence $u_{m_k}\in\mathcal H^q_{b, m_k}(X)$ with $\|u_{m_k}\|=1$ such that
\begin{equation}\label{b6}
\limsup\limits_{m\rightarrow\infty}m^{-(n-1)}S^q_{m, J}(x_0)=\lim_{k\rightarrow\infty}m_k^{-(n-1)}|u_{m_k, J}(x_0)|^2.
\end{equation}
$Tu_{m_k}=im_k u_{m_k}$ implies that on $D$, we have $u_{m_k}=\tilde u_{m_k}(z)e^{im_k\theta}$. Since $\Box^{(q)}_{b, m_k}(u_{m_k})=0$, from Lemma~\ref{b1} we have
\begin{equation}\label{b7}
\Box^{(q)}_{2m_k\varphi}(e^{m_k\varphi}\tilde u_{m_k}(z))=0.
\end{equation}
Moreover,
\begin{equation}\label{b8}
\begin{split}
\int_{\tilde D}|e^{m_k\varphi}\tilde u_{m_k}(z)|^2&e^{-2m_k\varphi(z)}\lambda(z)dv(z)=\int_{\tilde D}|\tilde u_{m_k}(z)|^2\lambda(z)dv(z)\\
&=\frac{1}{2\delta}\int_D|u_{m_k}|^2\lambda(z)dv(z)d\theta\leq\frac{1}{2\delta}.
\end{split}
\end{equation}
Similarly, set
$\tilde v_{(m_k)}(z)=m_k^{\frac{-(n-1)}{2}}e^{m_k\varphi(\frac{z}{\sqrt m_k})}\tilde u_{m_k}(\frac{z}{\sqrt m_k})$. Then from (\ref{s}), (\ref{b7}) and (\ref{b8}) we have
\begin{equation}\label{h1}
\Box^{(q)}_{(m_k)}\tilde v_{(m_{k})}=0~\text{on}~\tilde D_{\log m_k}, \end{equation} and
\begin{equation}\label{h2}
\int_{\tilde D_{\log m_k}}|\tilde v_{(m_k)}(z)|_{F_{m_k}^\ast}^2e^{-2m_k\varphi(\frac{z}{\sqrt m_k})}\lambda(\frac{z}{\sqrt m_k})dv(z)\leq\frac{1}{2\delta}.
\end{equation}
For any $r>0$ with $\tilde D_r\subset\tilde D_{\log m}$ when $m>>1$, by Garding's inequality we have
\begin{equation}\label{k5}
\|\tilde v_{(m_k)}\|^2_{2m_kF_{m_k}^\ast\varphi, s+2, \tilde D_{r}}\leq C_{s, r}\left(\|\tilde v_{(m_k)}\|^2_{2m_kF_{m_k}^\ast\varphi, \tilde D_{2r}}+\|\Box^{(q)}_{(m_k)}\tilde v_{(m_k)}\|^2_{2m_kF_{m_k}^\ast\varphi, s, \tilde D_{2r}}\right),
\end{equation}
where $C_{r,s}>0$ is a constant independent of $m_k$.
Combining (\ref{h1}), (\ref{h2}) and (\ref{k5}) we have $\|\tilde
v_{(m_k)}\|^2_{2mF_m^\ast\varphi, s+2, \tilde D_{r}}\leq C_{r, s,
\delta}$, where $C_{r,s,\delta}>0$ is a constant independent of $m_k$. We extend $\tilde v_{(m_k)}$ to $\mathbb C^{n-1}$ by zero outside $\tilde D_{\log m_k}$ still denoted by $\tilde v_{(m_k)}$. By  Sobolev compact embedding theorem, there exists a subsequence of $\{\tilde v_{(m_k)}(z)\}$ which is denoted by $\{v_{(m_{k_j})}(z)\}$ such that
\begin{equation}\label{b9}
\tilde v_{(m_{k_j})}\rightarrow v=\sideset{}{'}\sum_{|J|=q}v_J(z)d\overline z^J\in\Omega^{0, q}(\mathbb C^{n-1})~\text{in }~C^\infty(K)~\text {topology},~\ \ \forall K\Subset\mathbb C^{n-1}.
\end{equation}
From \eqref{h1}, \eqref{h2}, \eqref{b9} and \eqref{j7}, we can check that
\begin{equation}\label{e-gue150618a}
\Box^{(q)}_{2\Phi_0}v=0
\end{equation}
and
\begin{equation}\label{c1}
\int_{\mathbb C^{n-1}}|v(z)|^2e^{-2\Phi_0(z)}dv(z)\leq\frac{1}{2\delta}.
\end{equation}
Recall that $\Phi_0(z)=\sum_{j=1}^{n-1}\lambda_j|z_j|^2.$
Combining \eqref{e-gue150618a} and (\ref{c1}), we have
\begin{equation}\label{e-gue150618f}
|v_J(0)|^2\leq\frac{1}{2\delta}S^q_{\mathbb C^{n-1}, J}(0).
\end{equation}
Here, $S^q_{J, \mathbb C^{n-1}}(0)$ is the extremal function along the direction $d\overline z^J$ on  the model space $\mathbb C^{n-1}$ with respect to complex Laplacian $\Box^{(q)}_{2m\Phi_0}$, that is
\[S^q_{J, \mathbb C^{n-1}}(0)=\sup\{|u_J(0)|^2: u\in\Omega^{0,q}(\mathbb C^{n-1}), \Box^{(q)}_{2\Phi_0}u=0, \int|u|^2e^{-2\Phi_0(z)}dv(z)=1\}.\]
From (\ref{b6}), \eqref{b9} and \eqref{e-gue150618f}, we have
\begin{equation}\label{c2}
\limsup\limits_{m\rightarrow\infty}m^{-(n-1)}S^q_{m, J}(x_0)=\lim_{j\rightarrow\infty}m_{k_j}^{-(n-1)}|u_{m_{k_j}, J}(x_0)|^2=\lim_{j\rightarrow\infty}|\tilde v_{(m_{k_j})}(0)|^2=|v_J(0)|^2\leq\frac{1}{2\delta}S^q_{\mathbb C^{n-1}, J}(0).
\end{equation}
From \eqref{c2} and Lemma~\ref{l-b6}, we deduce that
\begin{equation}\label{e-gue150618fI}
\limsup\limits_{m\rightarrow\infty}m^{-(n-1)}\Pi^q_m(x_0)\leq\sideset{}{'}\sum_{|J|=q}\limsup\limits_{m\rightarrow\infty}m^{-(n-1)}S^q_{m, J}(x_0)\leq\frac{1}{2\delta}\sum\nolimits_{|J|=q}^\prime S^q_{J, \mathbb C^{n-1}}(0).
\end{equation}
By Proposition 4.3 in \cite{Be04}, we have that
\begin{equation}\label{e-gue150618fII}
\sum\nolimits_{|J|=q}^\prime S^q_{J, \mathbb C^{n-1}}(0)=\frac{1}{(2\pi)^{n-1}}|2\lambda_1\cdots2\lambda_{n-1}|\cdot 1_{X(q)}(x_0).
\end{equation}
From \eqref{e-gue150618fI} and \eqref{e-gue150618fII}, we have
\begin{equation}\label{g3}
\limsup_{m\rightarrow\infty}m^{-(n-1)}\Pi^q_m(x_0)\leq\frac{1}{2\delta}\cdot\frac{1}{2\pi^{n-1}}|\det \mathcal L_{x_0}|\cdot 1_{X(q)}(x_0).
\end{equation}
When $x_0\in X_k$, by Lemma~\ref{l-gue150616}, $\delta$ can be chosen to equal to $\frac{\pi}{k}-\epsilon$, for every $\epsilon>0$. From this observation and (\ref{g3}), we deduce that
\begin{equation}\label{g4}
\limsup_{m\rightarrow\infty}m^{-(n-1)}\Pi^q_m(x)\leq\frac{k}{2\pi}\cdot\frac{1}{2\pi^{n-1}}|\det \mathcal L_x|\cdot 1_{X(q)}(x), \forall x\in X_k
\end{equation}
and Theorem~\ref{m} follows then.

From Lemma~\ref{g2}, Theorem~\ref{m} and by Fatou's lemma we obtain the weak Morse inequalities and get the conclusion of Theorem~\ref{a4}.
\end{proof}

Now we are going to prove Theorem~\ref{o} and the strong Morse inequalities.

\subsection{Proofs of Theorem~\ref{o} and the strong Morse inequalities}

In this section, we will establish the strong Morse inequalities on CR manifolds with transversal CR $S^1$-action. We first recall some well known facts. From Theorem~\ref{gI}, we know that $\Box^{(q)}_{b, m}$ has discrete spectrum, each eigenvalues occurs with finite multiplicity and all the eigenforms are smooth. For $\sigma\in\mathbb R$, let $\mathcal H^q_{b, m, \leq\sigma}(X)$ be defined as in (\ref{c4}). Similarly, let $\mathcal H^q_{b, m, >\sigma}(X)$ denote the space spanned by the eigenforms of $\Box^{(q)}_{b, m}$ whose eigenvalues are $>\sigma$.

Let $Q_{b, m}$ be the Hermitian form on $\Omega^{0, q}_m(X)$ defined for $u, v\in\Omega^{0, q}_m(X)$ by
\begin{equation*}
Q_{b, m}(u, v)=(\overline\partial_b u|\overline\partial_b u)+(\overline\partial_b^\ast u|\overline\partial_b^\ast v)+(u|v)=(\Box^{(q)}_{b, m}u|v)+(u, v).
\end{equation*}
Let $\overline{\Omega^{0, q}_m(X)}$ be the completion of $\Omega^{0, q}_m(X)$ under the $Q_{b, m}$ in $L^2_{(0, q), m}(X)$. For $\lambda>0$, we have the orthogonal spectral decomposition with respect to $Q_{b, m}$
\begin{equation}
\overline{\Omega^{0, q}_m(X)}=\mathcal H^q_{b, m, \leq\sigma}\bigoplus\overline{\mathcal H^q_{b, m, >\sigma}(X)},
\end{equation}
where $\overline{\mathcal H^q_{b, m, >\sigma}(X)}$ is the completion of $\mathcal H^q_{b, m, >\sigma}(X)$ under $Q_{b, m}$ in $L^2_{(0, q), m}(X)$. For the proof of Theorem~\ref{o}, we need the following

\begin{proposition}\label{nn}
For any $p\in X(q)\cap X_{\rm reg}$, there exists $\alpha_m\in\Omega^{0, q}_m(X)$ such that
\begin{equation}
\begin{split}
&(1) ~\lim\limits_{m\rightarrow\infty}m^{-(n-1)}|\alpha_m(p)|^2=\frac{1}{2\pi^n}|\det \mathcal L_p|.\\
&(2) ~\lim_{m\rightarrow\infty}\|\alpha_m\|^2=1.\\
&(3) ~\lim_{m\rightarrow\infty}\left\|\left(m^{-1}\Box^{(q)}_{b, m}\right)^k\alpha_m\right\|=0, \forall k\in\mathbb N.\\
&(4) ~\text{There exists}~ \delta_m~\text{independent of}~p, \delta_m\rightarrow
0~\text{such that}\\
&\left(m^{-1}\Box^{(q)}_{b, m}\alpha_m\big|\alpha_m\right)\leq\delta_m.
\end{split}
\end{equation}
\end{proposition}

We now fix $p\in X(q)\cap X_{\rm reg}$. Let $D=\tilde D\times(-\pi,\pi)$ be a canonical local patch with canonical coordinates $(z,\theta,\varphi)$ such that $(z,\theta,\varphi)$ is trivial at $p$. We take $D=\{(z,\theta)\in\mathbb C^{n-1}: |z|<\varepsilon, |\theta|<\pi\}=\tilde D\times(-\pi,\pi)$. By Lemma~\ref{l-gue150615}, this is always possible.
Until further notice, we will work with $(z,\theta,\varphi)$ and we will use the same notations as in Section 1.4.
Before the proof of Proposition~\ref{nn}, we claim that one can find
$u(z)\in\Omega^{0, q}(\mathbb C^{n-1})$ such that
\begin{equation}\label{c6}
\begin{split}
\Box^{(q)}_{2\Phi_0}&u(z)=0,
\int_{\mathbb C^{n-1}}|u(z)|^2e^{-2\Phi_0(z)}dv(z)=\frac{1}{2\pi},\\
&\text{and}~
|u(0)|^2=\frac{1}{2\pi^n}|\lambda_1(p)\cdots\lambda_{n-1}(p)|.
\end{split}
\end{equation}
Recall that $\Phi^{(q)}_{2\Phi_0}$ is given by \eqref{e-gue150617I}.
Proof of the claim: We assume that the first $q$ eigenvalues of the Levi-form are negative, that is, $\lambda_1\leq\cdots\leq\lambda_q<0<\lambda_{q+1}\leq\cdots\leq\lambda_{n-1}.$ Set
\begin{equation}\label{e-gue150619}
u(w)=\left(\frac{|2\lambda_1\cdots2\lambda_{n-1}|}{(2\pi)^{n-1}}\cdot\frac{1}{2\pi}\right)^{\frac12}
e^{\sum_{j=1}^q\lambda_j|w_j|^2}d\overline w_1\wedge\cdots\wedge d\overline w_q.
\end{equation}
It is easy to check that the form $u(w)$ satisfies the claim. Now we are going to prove Proposition~\ref{nn}.

\begin{proof}
We choose cut-off function $\chi$ such that $\chi(z)\in C_0^\infty(\mathbb C^{n-1})$ with $\chi\equiv1$ in a neighborhood of $\overline{D_{\frac12}}$ and supp$\chi\Subset D_1$. Here, $D_r=\{z\in\mathbb C^{n-1}: |z_1|<r, \cdots, |z_{n-1}|<r\}$. Choose  a function $\eta(t)\in C^\infty(\mathbb R)$ satisfying $0\leq \eta(t)\leq 1$ such that $\eta(t)\equiv1$ when $t\geq \pi^2$ and $\eta(t)\equiv0$ when $t<\frac{\pi^2}{4}.$  Set $\eta_m(\theta)=\eta((\pi^2-\theta^2)\log^2 m), m\in\mathbb N.$ Then $\eta_m(\theta)$ is a family of cut-off functions with supp$\eta_m\Subset(-\pi, \pi)$. Moreover, we have that $\lim\limits_{m\rightarrow\infty}\eta_m(\theta)=1, ~a.e.~\theta\in(-\pi, \pi)$ and $|\eta_m|\leq1, |\eta_m^\prime(\theta)|=O(\log^2 m), |\eta_m^{''}(\theta)|=O(\log^4m)$. Define
\begin{equation}\label{c5}
u_m(z, \theta)=m^{\frac{n-1}{2}}u(\sqrt mz)e^{-m\varphi(z)}\chi\left(\frac{\sqrt mz}{\log m}\right)\eta_m(\theta)e^{im\theta},\end{equation}
where $u(z)\in\Omega^{0,q}(\mathbb C^{n-1})$ is as in \eqref{e-gue150619}.
Then $u_m\in\Omega^{0, q}(X)$ with supp$u_m\Subset D$. Set $\alpha_m=Q^{(q)}_m u_m$. Then on $D$ we have
\begin{equation}
\alpha_m(z, \theta)=\frac{1}{2\pi}\int_{-\pi}^{\pi}u_m(z, t)e^{-imt}dt e^{im\theta}.
\end{equation}
From (\ref{c5}) we have
\begin{equation}
\begin{split}
\|u_m\|^2
&=\int_Xm^{n-1}|u(\sqrt mz)|^2e^{-2m\varphi(z)}\chi^2\left(\frac{\sqrt mz}{\log m}\right)\eta_m^2(\theta)\lambda(z)dv(z)d\theta\\
&=\int_{-\pi}^{\pi}\eta_m^2(\theta)d\theta\int_{D_{\frac{\log m}{\sqrt m}}}m^{n-1}|u(\sqrt mz)|^2e^{-2m\varphi(z)}\chi^2\left(\frac{\sqrt mz}{\log m}\right)\lambda(z)dv(z)\\
&\leq 2\pi\int_{D_{\frac{\log m}{\sqrt m}}}m^{n-1}|u(\sqrt mz)|^2e^{-2m\varphi(z)}\chi^2\left(\frac{\sqrt mz}{\log m}\right)\lambda(z)dv(z).
\end{split}
\end{equation}
Taking limits as $m\rightarrow\infty$ and from the construction of $u(z)$ in (\ref{c6}), we have
\begin{equation}\label{c9}
\limsup_{m\rightarrow\infty}\|u_m\|^2\leq 2\pi\int_{\mathbb C^{n-1}}|u(z)|^2e^{-2\Phi_0(z)}dv(z)=2\pi\times\frac{1}{2\pi}=1.
\end{equation}
Since on $D$
\begin{equation}\label{c7}
\begin{split}
\alpha_m(z, \theta)&=Q_mu_m(z,\theta)=\frac{1}{2\pi}\int_{-\pi}^{\pi}u_m(z, t)e^{-imt}dte^{im\theta}\\
&=\frac{1}{2\pi}\int_{-\pi}^{\pi}m^{\frac{n-1}{2}}u(\sqrt mz)e^{-m\varphi(z)}\chi(\frac{\sqrt mz}{\log m})\eta_m(t)e^{imt}e^{-imt}dte^{im\theta}\\
&=\Bigr(\frac{1}{2\pi}\int_{-\pi}^{\pi}\eta_m(t)dt\Bigr)m^{\frac{n-1}{2}}u(\sqrt mz)e^{-m\varphi(z)}\chi\left(\frac{\sqrt mz}{\log m}\right)e^{im\theta}\\
&=c_mm^{\frac{n-1}{2}}u(\sqrt mz)e^{-m\varphi(z)}\chi\left(\frac{\sqrt mz}{\log m}\right)e^{im\theta}.
\end{split}
\end{equation}
Here $c_m=\frac{1}{2\pi}\int_{-\pi}^{\pi}\eta_{m}(t)dt$. Then by Fatou's lemma, we get $\lim\limits_{m\rightarrow\infty}c_m=1$. We have
\begin{equation}\label{e3}
m^{-(n-1)}|\alpha_m(p)|^2=m^{-(n-1)}|\alpha_m(0,0)|^2=c^2_m|u(0)|^2=c_m^2\frac{|\lambda_1(p)\cdots\lambda_{n-1}(p)|}{2\pi^n}.
\end{equation}
Taking limits in \eqref{e3} as $m\rightarrow\infty$, we get the conclusion of the first part of Proposition~\ref{nn}. From (\ref{c7}), we have
\begin{equation}\label{c8}
\begin{split}
&\lim_{m\rightarrow\infty}\int_D|\alpha_m(z, \theta)|^2\lambda(z)dv(z)d\theta\\
&=\lim_{m\rightarrow\infty}2\pi\int_{\tilde D}|c^2_m|m^{n-1}|u(\sqrt mz)|^2e^{-2m\varphi(z)}\chi^2\left(\frac{\sqrt mz}{\log m}\right)\lambda(z)dv(z)\\
&=2\pi\int_{\mathbb C^{n-1}}|u(z)|^2e^{-2\Phi_0(z)}dv(z)=2\pi\times\frac{1}{2\pi}=1.
\end{split}
\end{equation}
This implies that $\liminf\limits_{m\rightarrow\infty}\|\alpha_m\|^2\geq1$. From (\ref{c9}) and the definition of $\alpha_m$ we have $\|\alpha_m\|^2\leq\|u_m\|^2\leq1$ which implies that $\limsup\limits_{m\rightarrow\infty}\|\alpha_m\|^2\leq 1.$ Thus we have $\lim\limits_{m\rightarrow\infty}\|\alpha_m\|^2=1.$ Thus we get the conclusion of the second part of  Proposition~\ref{nn}. Now we postpone and state the following lemma

\begin{lemma}\label{g5}
\begin{equation}\label{d2}
\begin{split}
\frac{1}{m}\Box^{(q)}_{b}u_m&=\frac{1}{m}\Box^{(q)}_b\left[m^{\frac{n-1}{2}}u(\sqrt mz)e^{-m\varphi(z)}\chi\left(\frac{\sqrt mz}{\log m}\right)e^{im\theta}\eta_m(\theta)\right]\\
&=\frac{1}{m}m^{\frac{n-1}{2}}\Box^{(q)}_b\left[u(\sqrt mz)e^{-m\varphi(z)}\chi\left(\frac{\sqrt mz}{\log m}\right)e^{im\theta}\right]\eta_m(\theta)+\varepsilon_m,
\end{split}
\end{equation}
where $\|\varepsilon_m\|\leq\delta_m$, $\delta_m$ is a sequence independent of $p$ with $\delta_m\To0$ as $m\rightarrow\infty$.
\end{lemma}

\begin{proof}
Let $\{e^j\}_{j=1}^{n-1}$ be the orthonormal frame of $T^{\ast 0, 1}X$ over $D$ given in Lemma~\ref{b5}. Let $\{\overline U_j\}_{j=1}^{n-1}$ be the dual frame of $\{e^j\}_{j=1}^{n-1}$ with respect to the given $T$-rigid Hermitian metric on $\mathbb CTX.$ Then on $D$
\begin{equation}
\begin{split}
\overline U_j=\frac{\partial}{\partial\overline z_j}-i\lambda_jz_j\frac{\partial}{\partial\theta}+O(|z|^2)
\frac{\partial}{\partial\theta}, j=1, \ldots, n-1.
\end{split}
\end{equation}
By a direct calculation(see Proposition 2.3 in \cite{H08})
\begin{equation}
\Box^{(q)}_b=\sum_{j=1}^{n-1}\overline U_j^\ast\overline U_j+\sum_{j, k=1}^{n-1}e^j\wedge(e^k\wedge)^\ast \circ[\overline U_j, \overline
U_k^\ast]+\varepsilon(\overline U)+\varepsilon (\overline
U^\ast)+~\text{zero order terms},
\end{equation}
where $U_j^\ast$ is the formal adjoint of $U_j$, $\varepsilon(\overline U)$ denotes the remainder terms of the form $\sum\limits_{k=1}^{n-1}a_k(z, \theta)\overline U_k$ with $a_k$ smooth and similarly for $\varepsilon (\overline U^\ast)$. Then by a direct calculation we have
\begin{equation}\label{g7}
\begin{split}
\frac{1}{m}\Box^{(q)}_bu_m&=\frac{1}{m}m^{\frac{n-1}{2}}\Box^{(q)}_b\left[u(\sqrt mz)e^{-m\varphi(z)}\chi\left(\frac{\sqrt mz}{\log m}\right)e^{im\theta}\right]\eta_m(\theta)\\
&+\frac{1}{m}\left(\varepsilon(\overline U^\ast)u_m(z, \theta)\right)\eta_m^\prime(\theta)O(|z|)+\frac{1}{m}
\left(\varepsilon(\overline U)u_m(z, \theta)\right)\eta_m^\prime(\theta)O(|z|)\\
&+\frac{1}{m}u_m(z, \theta)\left[\eta_m^\prime(\theta)O(1)+\eta_m^\prime(\theta)O(|z|)+\eta^{\prime\prime}_m(\theta)O(|z|^2)
\right].\\
&=\frac{1}{m}m^{\frac{n-1}{2}}\Box^{(q)}_b\left[u(\sqrt mz)e^{-m\varphi(z)}\chi\left(\frac{\sqrt mz}{\log m}\right)e^{im\theta}\right]\eta_m(\theta)+\varepsilon_m
\end{split}
\end{equation}
Here, we have used $\varepsilon_m$ to denote the remaining terms of \eqref{g7}. Then by the construction of $\eta_m$ we can check that $\varepsilon_m=O(\frac{(\log m)^{\alpha}}{m^{\beta}})$ where $\alpha, \beta$ are positive constants. Thus the lemma follows.
\end{proof}
Now we are going to prove the third part of  Proposition~\ref{nn}, we only prove it when $k=1$ and the other cases are similar.
From Lemma~\ref{b1} we have
\begin{equation}
\Box^{(q)}_b\left[u(\sqrt mz)\chi\left(\frac{\sqrt mz}{\log m}\right)e^{-m\varphi(z)}e^{im\theta}\right]=e^{im\theta}e^{-m\varphi}
\Box^{(q)}_{2m\varphi}
\left[u(\sqrt mz)\chi\left(\frac{\sqrt mz}{\log m}\right)\right].
\end{equation}
From \eqref{s}, \eqref{j7} and $\Box^{(q)}_{2\Phi_0}u=0$, it is straightforward to check that
\begin{equation}\label{d1}
\begin{split}
&\int_X\left|\frac{1}{m}m^{\frac{n-1}{2}}\Box^{(q)}_b\left[u(\sqrt mz)e^{-m\varphi(z)}\chi\left(\frac{\sqrt mz}{\log m}\right)e^{im\theta}\right]\eta_m(\theta)\right|^2dv_X\\
&=\int\left|\frac{1}{m}m^{\frac{n-1}{2}}e^{im\theta}\Box^{(q)}_{2m\varphi}
\left[u(\sqrt mz)\chi\left(\frac{\sqrt mz}{\log m}\right)\right]\eta_m(\theta)\right|^2e^{-2m\varphi(z)}\lambda(z)dv(z)d\theta\\
&=\int^\pi_{-\pi}|\eta_m(\theta)|^2d\theta\int\left|\frac{1}{m}m^{\frac{n-1}{2}}e^{im\theta}\Box^{(q)}_{2m\varphi}
\left[u(\sqrt mz)\chi\left(\frac{\sqrt mz}{\log m}\right)\right]\right|^2e^{-2m\varphi(z)}\lambda(z)dv(z)\\
&=\int^\pi_{-\pi}|\eta_m(\theta)|^2d\theta\int\left|\Box^{(q)}_{(m)}
\left[u(z)\chi\left(\frac{z}{\log m}\right)\right]\right|^2e^{-2mF^*_m\varphi(z)}\lambda(\frac{z}{\sqrt{m}})dv(z)\\
&\leq 2\pi\int\left|\Box^{(q)}_{(m)}\left[u(z)\chi\left(\frac{z}{\log m}\right)\right]\right|^2e^{-2mF^*_m\varphi(z)}\lambda(\frac{z}{\sqrt{m}})dv(z)\leq\delta_m,
\end{split}
\end{equation}
where $\delta_m>0$ is a sequence independent of $p$ with $\lim_{m\To\infty}\delta_m=0$.
Combining (\ref{d2}), (\ref{d1}) and notice that $\|m^{-1}\Box^{(q)}_b\alpha_m\|\leq\|m^{-1}\Box^{(q)}_bu_m\|$ we get the conclusion of the third part of this proposition.
(2) in Proposition~\ref{nn} and \eqref{d1} imply (4) in this proposition.
\end{proof}

Now we are going to prove Theorem~\ref{o}. The proof of \eqref{e-gue150619d} is essentially the same as the proof of \eqref{e-gue150619dI}. Therefore we omit the detail. Let $\alpha_m$ be the sequence we have chosen in Proposition~\ref{nn}. Then $\alpha_m=\alpha_{m, 1}+\alpha_{m, 2}, \alpha_{m, 1}\in\mathcal H^q_{b, m, \leq mv_m}(X), \alpha_{m, 2}\in\overline{\mathcal H^q_{b, m, >mv_m}(X)}.$ Since
\begin{equation}\label{e-gue150620fb}
\|\alpha_{m, 2}\|^2=(\alpha_{m, 2}|\alpha_{m, 2})\leq\frac{1}{mv_m}\left(\Box_{b, m}^{(q)}\alpha_{m, 2}\big|\alpha_{m, 2}\right)
=\frac{1}{v_m}\left(\frac{1}{m}\Box^{(q)}_{b, m}\alpha_m\Big|u_{m, 2}\right)\leq\frac{\delta_m}{v_m}\rightarrow0.
\end{equation}
From \eqref{e-gue150620fb} and (2) in Proposition~\ref{nn}, we get
\begin{equation}
\lim_{m\rightarrow\infty}\|\alpha_{m, 1}\|=1.
\end{equation}
Now we claim that
\begin{equation}\label{e5}
\lim_{m\rightarrow\infty}m^{-(n-1)}|\alpha_{m, 2}(p)|^2=0.
\end{equation}
On $D$, we write $ \alpha_{m, 2}(z, \theta)=\tilde\alpha_{m, 2}(z)e^{im\theta}.$ Set $\beta_{m, 2}(z)=\tilde \alpha_{m, 2}(z)e^{m\varphi(z)}$. Then
\begin{equation}\label{d4}
\begin{split}
&\lim_{m\rightarrow\infty}m^{-(n-1)}|\alpha_{m, 2}(p)|^2=\lim_{m\rightarrow\infty}m^{-(n-1)}|\tilde\alpha_{m ,2}(0)|^2\\
=&\lim_{m\rightarrow\infty}m^{-(n-1)}|\beta_{m, 2}(0)|^2
=\lim_{m\rightarrow\infty}|\beta_{(m), 2}(0)|^2.
\end{split}
\end{equation}
Here we used the notation $|\beta_{(m), 2}(z)|^2=m^{-(n-1)}|\beta_{m, 2}(\frac{z}{\sqrt m})|^2.$

From Lemma~\ref{b1} we have
\begin{equation}\label{d3}
\Box^{(q)}_{b, m}\left(\alpha_{m, 2}\right)=e^{im\theta}e^{-m\varphi(z)}\Box^{(q)}_{2m\varphi}
\left(\tilde\alpha_{m, 2}(z)e^{m\varphi(z)}\right)
=e^{im\theta}e^{-m\varphi(z)}\Box^{(q)}_{2m\varphi}(\beta_{m, 2}).
\end{equation}
From (\ref{d3}) and using induction, we get on $D$
\begin{equation}\label{d6}
(\Box_b^{(q)})^k\alpha_{m, 2}=e^{im\theta}e^{-m\varphi}(\Box^{(q)}_{2m\varphi})^k(\beta_{m, 2}(z)).
\end{equation}
By Garding's inequality (see Proposition~\ref{k2}) and Sobolev embedding theorem, we see that
\begin{equation}\label{d8}
|\beta_{(m), 2}(0)|^2\leq C_{n, r}\left(\|\beta_{(m),2}\|^2_{2mF^*_m\varphi, D_r}+\|\Box^{(q)}_{(m)}\beta_{(m), 2}\|_{2mF^*_m\varphi, n, D_{r}}^2\right)
\end{equation}
for some $r>0$. Here $C_{n, r}$ is a constant independent of $p$ and $m$.
Now, we have
\begin{equation}\label{d9}
\|\beta_{(m),2}\|^2_{2mF^*_m\varphi, D_r}\leq\|\alpha_{m, 2}\|^2\rightarrow0.
\end{equation}
Moreover, from Garding's inequality and using induction (see Proposition~\ref{k2}), we have
\begin{equation}\label{e1}
\|\Box^{(q)}_{(m)}\beta_{(m), 2}\|_{2mF^*_m\varphi, n, D_{r}}^2\leq C^\prime\sum_{k=1}^{n+1}\left\|\left(\Box^{(q)}_{(m)}\right)^k\beta_{(m), 2}\right\|_{2mF^*_m\varphi, D_{r^\prime}}^2,
\end{equation}
for some $r^\prime>0$, where $C^\prime>0$ is a constant independent of $m$. From (\ref{d6}) and \eqref{s} we can check that for $k\in\mathbb N$,
\begin{equation}\label{e2}
\begin{split}
\left\|\left(\Box^{(q)}_{(m)}\right)^k\beta_{(m), 2}\right\|_{2mF^*_m\varphi, D_{r^\prime}}^2
&\leq C_1\left\|\frac{1}{m^k}\left(\Box^{(q)}_{b}\right)^k\alpha_{m, 2}\right\|^2\\
&\leq C_1\left\|\frac{1}{m^k}\left(\Box^{(q)}_{b}\right)^k\alpha_{m}\right\|^2\To0,
\end{split}
\end{equation}
where $C_1>0$ is a constant independent of $m$.
Combining (\ref{d8}), (\ref{d9}), (\ref{e1}) with (\ref{e2}), we have
$
\lim_{m\rightarrow\infty}|\beta_{(m), 2}(0)|^2=0.
$
From (\ref{d4}) we have
$
\lim_{m\rightarrow\infty}m^{-(n-1)}|\alpha_{m, 2}(p)|^2=0
$
and the claim (\ref{e5}) follows. From \eqref{e5} and (1) in Proposition~\ref{nn}, we conclude that
\begin{equation}
\lim_{m\rightarrow\infty}m^{-(n-1)}|\alpha_{m, 1}(p)|^2=\frac{|\lambda_1(p)\cdots\lambda_{n-1}(p)|}{2\pi^n}.
\end{equation}
Now,
\begin{equation}\label{e6}
m^{-(n-1)}\Pi^q_{m, \leq mv_m}(0)\geq m^{-(n-1)}\frac{|\alpha_{m, 1}(p)|^2}{\|\alpha_{m, 1}\|^2}\rightarrow\frac{|\lambda_1(p)\cdots\lambda_{n-1}(p)|}{2\pi^n}.
\end{equation}
By a similar proof of (\ref{k6}), we have
\begin{equation}\label{e7}
\limsup_{m\rightarrow\infty}m^{-(n-1)}\Pi^q_{m, \leq mv_m}(p)\leq \frac{|\lambda_1(p)\cdots\lambda_{n-1}(p)|}{2\pi^n}.
\end{equation}
Combining (\ref{e6}) with (\ref{e7}), we get the conclusion
of Theorem~\ref{o}.
\section{Appendix}
\subsection{Proof of Lemma~\ref{g2}}
\begin{proof}
Set $X_j=\{x\in X: e^{i\frac{2\pi}{j}}\circ x=x~\text{and}~\forall~ 0<|\theta|<\frac{2\pi}{j}, e^{i\theta}x\neq x\}.$ We call such $x$ the points in $X$ with period $\frac{2\pi}{j}.$ Then $X_{\rm reg}=X_1$ by definition. There are only finite $X_j, 1\leq j\leq m$ such that $X=\bigcup_{j=1}^m X_j.$ Then $X_j\cap X_k=\emptyset, \forall j\neq k$.
Now we are going to show that $\bigcup_{j=2}^mX_j$ is a closed subset of $X$. We assume there exists a sequence $\{x_k\}\subset\bigcup_{j=2}^mX_j$ such that $x_k\rightarrow x_0$. W.L.O.G, we assume that the $\{x_k\}\subset X_j$ for some $j\geq 2$. Then we have $e^{i\frac{2\pi}{j}}\circ x_k=x_k$. Taking limits as $k\rightarrow\infty$ we have $e^{i\frac{2\pi}{j}}\circ x_0=x_0$. By definition, $x_0\not\in X_{\rm reg}$. Thus $x_0\in \bigcup_{j=2}^mX_j$. This means that $\bigcup_{j=2}^mX_j$ is a closed subset of $X$ and the complement $X_{\rm reg}$ is an open subset of $X$.

Second, we are going to check that the measure of $X\setminus X_{\rm reg}$ is zero. Set $Y_j=\{x\in X: e^{i\frac{2\pi}{j}}\circ x=x\}, 2\leq j\leq m$. Obviously that $Y_j$ is a closed subset of $X$ and $X_j\subset Y_j, 2\leq j\leq m$. Now we will show that the measure of $Y_j, 2\leq j\leq m$ is zero and for convenient we only show that the measure of $Y_2$ is zero. We use $m(Y_j)$ to denote the measure of $Y_j$ for $2\leq j\leq m$. If $Y_2=\emptyset$, we have $m(Y_2)=0$. Now we assume that $Y_2\neq\emptyset.$ For any $p\in Y_2$, we have $e^{i\pi}\circ p=p$. With the rigid Hermitian metric on $X$, it is easy to check that the map $e^{i\pi}: X\rightarrow X$ is an isometrically CR isomorphism. Since $e^{i\pi}\circ p=p$ we have $de^{i\pi}: T_pX\rightarrow T_pX.$ Here $T_pX$ is the tangent space of $X$ at $p$. There exists a small neighborhood $U_{o_p}$ of $o_p\in T_pX$ such that the exponential map
\begin{equation}\label{h3}
\exp_p: U_{o_p}\rightarrow \exp_p(U_{o_p}):=V_p\subset X\end{equation}
is a diffeomorphism. Then for any $q\in Y_2\cap V_p$, there exists a vector $Z_q\in U_{o_p}$ such that $\exp(Z_q)=q$. Since $e^{i\pi}\circ q=q$, we have that $e^{i\pi}(\exp_p(Z_q))=q=\exp_p(Z_q)$. The isometric map $e^{i\pi}: X\rightarrow X$ implies the commutation between $e^{i\pi}$ and the exponential map and we have that
\begin{equation}\label{h4}
\exp_p\circ de^{i\pi}(Z_q)=e^{i\pi}\circ\exp_p(Z_q)=q=\exp_p(Z_q).
\end{equation}
Since $\|de^{i\pi}(Z_q)\|=\|Z_q\|$, we have that $de^{i\pi}(Z_q)\in U_{o_p}$. Combining with (\ref{h3}), we get $de^{i\pi}(Z_q)=Z_q$. This means  that $Z_q$ is a fixed point of the linear map $de^{i\pi}: T_p X\rightarrow T_pX.$ Set $H=\{Z\in T_p X: de^{i\pi}Z=Z\}$. By (\ref{h3}) and (\ref{h4}) we have that
\begin{equation}\label{h5}
\exp_p(U_{o_p}\cap H)=V_p\cap Y_2.
\end{equation}
Since $Y_2$ is a closed subset of $X$, From (\ref{h5}) we have that $H$ must be a proper linear subspace of $T_pX$. Then (\ref{h5}) implies that $m(Y_2)=0$. Similarly, we have $m(Y_j)=0, \forall 2\leq j\leq m$. From $X_j\subset Y_j, 2\leq j\leq m$, we have that $m(X_j)=0, 2\leq j\leq m.$
Moreover (\ref{h5}) implies that $Y_2$ is a nowhere dense subset of $X$, similarly, $Y_j, 2\leq j\leq m$ are nowhere dense subset of $X$. Since $X_j\subset Y_j, 2\leq j\leq m$, we have that $X_{\rm reg}$ is a dense subset of $X$.
\end{proof}

\begin{center}
{\bf Acknowledgement}
\end{center}

The authors would like to express their gratitude to Professor Xiaochun Rong for the helpful communications on group actions and  Hendrik Herrmann for useful discussion in this work. The authors are grateful to Professor Marinescu for comments and
useful suggestions on an early draft of the manuscript. We also thank the referee for many detailed remarks that have helped to improve the presentation.

\bibliographystyle{amsalpha}

\end{document}